\documentclass[11pt]{amsart}
\usepackage{tikz}
\usetikzlibrary{patterns.meta}
\usepackage{amsmath}
\usepackage{amssymb}
\usepackage{mathrsfs}
\usepackage{xcolor, enumitem, comment, todonotes}
\usepackage[all]{xy}
 \usepackage{url}
 \usepackage{amsthm}
 \usepackage{thmtools}
\usepackage{thm-restate}
\usepackage{hyperref,soul}

\theoremstyle{plain}
\newtheorem*{theorem*}{Theorem}
\newtheorem*{corollary*}{Corollary}

\newtheorem{theorem}{Theorem}[section]
\newtheorem{claim}[theorem]{Claim}

\newtheorem{lemma}[theorem]{Lemma}

\newtheorem{proposition}[theorem]{Proposition}
\newtheorem{corollary}[theorem]{Corollary}

\theoremstyle{definition}
\newtheorem{definition}[theorem]{Definition}
\newtheorem{example}[theorem]{Example}

\newtheorem{question}[theorem]{Question}

\theoremstyle{remark}
\newtheorem{remark}[theorem]{Remark}

\newtheorem{fact}[theorem]{Fact}

\def\l{{\langle}}
\def\r{{\rangle}}
  
\newcommand{\FIN}{\mathrm{FIN}}
\newcommand{\fin}{\mathrm{fin}}

\newcount\skewfactor
\def\mathunderaccent#1#2 {\let\theaccent#1\skewfactor#2
\mathpalette\putaccentunder}
\def\putaccentunder#1#2{\oalign{$#1#2$\crcr\hidewidth
\vbox to.2ex{\hbox{$#1\skew\skewfactor\theaccent{}$}\vss}\hidewidth}}

\def\smallbox#1{\leavevmode\thinspace\hbox{\vrule\vtop{\vbox
   {\hrule\kern1pt\hbox{\vphantom{\tt/}\thinspace{\tt#1}\thinspace}}
   \kern1pt\hrule}\vrule}\thinspace}

\DeclareMathOperator{\rng}{rng}

\title[Commutativity of products]{Commutativity of cofinal types}
\date{\today}

\author{Tom Benhamou}
\thanks{This research  was supported by the National Science Foundation under Grant
No. DMS-2346680}
\address[Benhamou]{Department of Mathematics, Rutgers University, New Brunswick ,NJ USA}
\email{tom.benhamou@rutgers.edu}


\begin{document}
\maketitle

\begin{abstract}
    We continue the study of the pseudo-intersection property with respect to an ideal introduced in \cite{TomNatasha2}.  Our theory applies to the study of the Tukey types of general sums of ultrafilters,  which, as evidenced by the results of this paper, can be quite complex. It also applies to construct a large class of ultrafilter $\mathcal{C}$ over $\omega$ such that any two ultrafilters $U,V\in \mathcal{C}$ commute; that is, $U\cdot V\equiv_T V\cdot U$.  The class $\mathcal{C}$ class contains most known cofinal types of ultrafilters on $\omega$. This is in sharp contrast to the Rudin-Keisler ordering. In the third part of this paper, we apply our results to study the class of ultrafilters Tukey above $\omega^\omega$. Specifically, we prove that ultrafilters without the $I$-p.i.p are always above $I^\omega$ and in particular non-$p$-points are Tukey above $\omega^\omega$. Finally, we introduce the hierarchy of $\alpha$-almost rapid ultrafilters. We prove that it is consistent for them to form a strictly wider class than the rapid ultrafilters, and give an example of a non-rapid $p$-point ultrafilter which is Tukey above $\omega^\omega$. This addresses and answers several questions from \cite{TomNatasha,TomNatasha2,Dobrinen/Todorcevic11,Milovich08}.
\end{abstract}
\section{Introduction}
The Tukey order stands out as one of the most studied orders of ultrafilters \cite{Milovich08,Dobrinen/Todorcevic11,Hrusak,Blass/Dobrinen/Raghavan15,Raghavan/Verner19,TomNatasha}. Its origins lie in the examination of Moore-Smith convergence, and it holds particular significance in unraveling the cofinal structure of the partial order $(U,\supseteq)$ of an ultrafilter.
Formally, given two posets, $(P,\leq_P)$ and $(Q,\leq_Q)$ we say that $(P,\leq_P)\leq_T (Q,\leq_Q)$ if there is map $f:Q\rightarrow P$, which is cofinal, namely, $f''\mathcal{B}$ is cofinal in $P$ whenever $\mathcal{B}\subseteq Q$ is cofinal. Schmidt \cite{Schmidt55} observed that this is equivalent to having a map $f:P\rightarrow Q$, which is unbounded, namely, $f''\mathcal{A}$ is unbounded in $Q$ whenever $\mathcal{A}\subseteq P$ is unbounded in $P$.
We say that $P$ and $Q$ are {\em Tukey equivalent},
and write
$P\equiv_T Q$,
if $P\leq_T Q$ and $Q\leq_T P$;  the equivalence class $[P]_T$ is called the \textit{Tukey type} or \textit{cofinal type} of $P$.

The scope of the study of cofinal types of ultrafilters covers several long-standing open problems such as: 
\begin{itemize}
    \item Isbell's problem\cite{Isbell65}: Is it provable within ZFC that a non-Tukey-top ultrafilter\footnote{A Tukey-top ultrafilter is an ultrafilter which is Tukey equivalent to the poset $([\mathfrak{c}]^{<\omega},\subseteq)$. Isbell \cite{Isbell65} constructed such ultrafilters in $ZFC$.} on $\omega$ exists?
    \item Kunen's Problem: Is it consistent that $\mathfrak{u}_{\aleph_1}<2^{\aleph_1}$? Namely, is it consistent to have a set $\mathcal{B}\subseteq P(\omega_1)$ of cardinality less than $2^{\aleph_1}$ which generates an ultrafilter?
\end{itemize} 
The Tukey order is also related to the Katovich problem. A systematic study of the Tukey order on ultrafilter over $\omega$, traces back to Isbell \cite{Isbell65}, later to Milovich \cite{Milovich08} and Dobrinen and Todorcevic \cite{Dobrinen/Todorcevic11}. Lately, Benhamou and Dobrinen  \cite{TomNatasha} extended this study to ultrafilters on cardinals greater than $\omega$. Over measurable cardinals, the Tukey order is connected to recent developments revolving the so-called Galvin property, studied by Abraham, Benhamou, Garti, Goldberg, Gitik, Hayut, Magidor, Poveda, Shelah and others \cite{AbrahamShelah1986,Garti2017WeakDA,GartiTilt,bgp,Non-GalvinFil,GalDet,OnPrikryandCohen,Parttwo,ghhm,TomGabe}; the Galvin property in one of its forms is equivalent to being Tukey-top as shown essentially by Isbell (in different terminology).  Moreover, being Tukey-top in the restricted class of $\kappa$-complete ultrafilters takes the usual studied forms of the Galvin property.

In this paper, we address the problem of commutativity of the Tukey types of Fubini products of ultrafilters $U,V$ over $\omega$ (Definition \ref{Definition: Product and power}), denoted by $U\cdot V$. This problem was suggested in \cite{TomNatasha2}, and was already partially addressed:
\begin{itemize}
    \item Dobrinen and Todorcevic \cite{Dobrinen/Todorcevic11} proved that if $U,V$ are rapid $p$-points then $U\cdot V\equiv_T V\cdot U$.
    \item Milovich \cite{Milovich12} extended this result to prove that if $U,V$ are just $p$-points, then $U\cdot V\equiv_T V\cdot U$.
    \item Benhamou and Dobrinen proved later that if $U,V$ are $\kappa$-complete ultrafilters over a measurable cardinal $\kappa$ then $U\cdot V\equiv_T V\cdot U$.
\end{itemize}
Our goal is to study larger classes where commutativity holds, these are called \textit{commutative classes} of ultrafilters. In that avenue, we prove the following:
\begin{theorem*}
    There is a commutative class of ultrafilter $\mathcal{C}$ which is closed under Fubini sums and includes all Tukey-top ultrafilters, $p$-points, stable ordered union ultrafilters, generic ultrafilters for $P(\omega)/I$ for a wide range of $I$'s. \footnote{$I$ should be simple, and $I^\omega\equiv_V I$, so for example $I$ can be $\fin^{\otimes\alpha}$.}
\end{theorem*}
The commutative class described in the previous theorem includes most known ultrafilters on $\omega$, and in fact it is open whether this class consistently excludes an ultrafilter  on $\omega$. 
The commutativity of cofinal types stands in sharp contrast to the Rudin-Keisler ordering which is known to be highly non commutative with respect to Fubini product\footnote{For example if $U,W$ are non-isomorphic Ramsey ultrafilters then $U\cdot W\not\equiv_{RK}W\cdot U$. Just otherwise, by a theorem of Rudin (see for example \cite[Thm. 5.5]{Kanamori1976UltrafiltersOA}, $U,W$ should be Rudin-Frol\'{i}k (and therefore Rudin-Keisler) comparable, contradicting the RK-minimality of Ramsey ultrafilters.}. On measurable cardinals, the situation is even more dramatic, due to a theorem of Solovey (see \cite[Thm. 5.7]{Kanamori1976UltrafiltersOA}) if $U,W$ are $\kappa$-complete ultrafilters on $\kappa$ the $U\cdot W\equiv_T W\cdot U$ if and only if $W\equiv_{RK} U^n$ for some $n$ or vise versa.
Recently, Goldberg \cite{goldberg2021products} examined situations of commutativity with respect to several product operations on countably complete ultrafilters.

The main idea that is used in the proof, is to analyze the cofinal types of ideals and filters connected to a given ultrafilter $U$. More specifically, we will exploit the idea of the \textit{pseudo-intersection property with respect to $I$} (Definition \ref{definition: the I-p.i.p}) which was introduced in \cite{TomNatasha2} and was used to prove that Milliken-Taylor ultrafilters and generic ultrafilters\footnote{For the definition of $I^{\otimes\alpha}$, see the paragraph before Fact \ref{Fact: dual operation}.} for $P(\omega^\alpha)/\fin^{\otimes\alpha}$ satisfy $U\cdot U\equiv_T U$. In fact we use the theory developed here to  prove some generalization of theorem regarding generic ultrafilters over $P(\omega^\alpha)/\fin^{\otimes\alpha}$:
\begin{theorem*}
    Let $I$ be a deterministic $\sigma$-ideal\footnote{A $\sigma$-ideal is an ideal such that $P(\omega)/I$ is $\sigma$-complete}. Then for every $\alpha<\omega_1$, and any $V$-generic ultrafilter for $P(\omega^\alpha)/I^{\otimes\alpha}$, $G\cdot G\equiv_T G$. 
\end{theorem*}
In ~\S\ref{section: p.i.p}, we provide a comprehensive study of this property. The main result of this section is
\begin{theorem*}
    Suppose that $\mathcal{A}$ is a discrete set of ultrafilters. Then for each $U\in \mathcal{A}$, $U$ has the $(\bigcap\mathcal{A})^*$-p.i.p.
\end{theorem*}

The property of ideals which will be in frequent use, is the property of being  \textit{deterministic} (Definition \ref{Def: deterministinc}). This property guarantees that whenever $I\subseteq J$, $I\leq_T J$.

We also investigate the Tukey type of ultrafilters of the form $\sum_UV_\alpha$. The only general result regarding the Tukey-type of such ultrafilters is due\footnote{Dobrinen and Todorcevic proved in for $\kappa=\omega$, but the proof for a general $\kappa$ appears in \cite{TomNatasha}.} to Dobrinen and Todorcevic \cite{Dobrinen/Todorcevic11} where they prove that if $U$ is an ultrafilter on $\kappa$, and $V_\alpha$ is a sequence of ultrafilters, then $\sum_UV_\alpha\leq_T U\times\prod_{\alpha<\kappa}V_\alpha$. It turns out that the Tukey class of such ultrafilters is much more complicated and the nice characterization we have for $U\cdot V\equiv_T U\times V^\omega$ is missing in the general case.  It is not hard to see that the ultrafilter $\sum_UV_\alpha$ is Tukey below each ultrafilter in the set $\mathcal{B}(U,\l V_\alpha\mid \alpha<\kappa\r)=\{U\times \prod_{\alpha\in X}V_\alpha\mid X\in U\}$. We prove that in some sense it is the greatest lower bound:
\begin{theorem*}
    For complete directed ordered set $\mathbb{P}$, $\mathbb{P}$ is uniformly below \footnote{See Definition \ref{Def: uniformly below}.} $\mathcal{B}(U,\l V_\alpha\mid \alpha<\kappa\r)$
     if and only if  $\sum_UV_\alpha\geq_T \mathbb{P}$.
\end{theorem*} By putting more assumptions on the sequence of ultrafilters, we are able to get nicer results. One of them is that under the assumption $$V_0\leq_T V_1\leq_T....$$ in which case we have that $\sum_U V_n\equiv_T U\times \prod_{n<\omega}V_n$. From this special case, we recover D. Milovich's formula $U\cdot V\equiv_T U\times V^{\omega}$ (See Theorem \ref{Thm: Milovich}). Another set of assumptions we consider is when for each $n<\omega$, $V_n\cdot V_n\equiv_T V_n\geq_T V_{n+1}$. These assumptions imply that $\sum_UV_n$ is a strict greatest least lower bound of $\mathcal{B}(U,\l V_n\mid n<\omega\r)$. This shows that the cofinal type of $\sum_UV_n$ is much more complicated than the cofinal type of $U\cdot V$.
We also provide an example (Prop. \ref{prop: example in full support}) where $U<_T\sum_U V_n<U\times \prod_{n\in X}V_n$ for every $X\in U$.

Finally, we study the class of ultrafilters $U$ such that $U\geq_T\omega^\omega$. The partial order $\omega^\omega$ appeared quite a bit in the literature \cite{Solecki/Todorcevic04,LOUVEAU1999171,Hrusak} in the context of the Tukey order on Borel ideals. In the context of general ultrafilters on $\omega$, it is known to be the immediate successor of the Tukey type $\omega$  \cite{LOUVEAU1999171}
\begin{theorem}[Louveau-Velickovic]\label{Thm: LouveauVel}
    If $I$ is any ideal such that $I<_T\omega^\omega$ then $I$ is countably generated.
\end{theorem}
On the other hand, among analytic ideals, it is known to be minimal \cite{TodorcevicDirSets85}(See also \cite[Thm 6.6]{Hrusak}):
\begin{theorem}[Todorcevic]\label{Thm: Todorcevic analytic p}
    Suppose that $I$ is an analytic $p$-ideal, then either $I$ is countable generated or $I\geq_T \omega^\omega$.
\end{theorem}
This was later improved by Solecki and Todorcevic \cite[Proposition 4.3]{Solecki/Todorcevic04} to show that if $I$ is analytic, not locally compact ideal, then $I\geq_T \omega^\omega$. 
This partial order came up in the work of Milovich  who asked \cite[Question 4.7]{Milovich08} if there is an ultrafilter $U$ over $\omega$ such that $(U,\supseteq)\equiv_T \omega^\omega$. We will observe that this was basically answered in \cite[Cor. 54]{Solecki/Todorcevic04}:
\begin{theorem}[Solecki-Todorcevic]\label{Thm: todorcevic and solecki}
     Suppose that $D$ is an ordered separable metric space such that the predecessors of each element form a compact set, and $E$ is a basic \footnote{For the definition of basic see \cite[\S 3]{Solecki/Todorcevic04}.} analytic order such that $D\leq_T E$, then $D$ is analytic.
\end{theorem}
Later, Dobrinen and Todorcevic contributed a great deal to the understanding of this class, in particular, they proved the following \cite[Thm. 35]{Dobrinen/Todorcevic11}:
\begin{theorem}[Dobrinen-Todorcevic]\label{Thm: Dobrinen and todorcevic rapid p-points}
The following are equivalent for a $p$-point:
\begin{enumerate}
    \item $U\cdot U\equiv_T U$.
    \item  $U\geq_T \omega^\omega$.
\end{enumerate}
\end{theorem}
They also proved that rapid ultrafilters are above $\omega^\omega$ and deduced that rapid $p$-points satisfy $U\cdot U\equiv_T U$.
This was lately improved by Benhamou and Dobrinenn \cite[Thm 1.18]{TomNatasha2} to the general setup of the $I$-p.i.p.
\begin{theorem}[Dobrinen-B.]\label{Thm: TomNatahsa equivalence for product}
Let $U$ be an ultrafilter.  Then the following are equivalent:
\begin{enumerate}
    \item $U\cdot U\equiv_T U$.
    \item There is an ideal $I\subseteq U^*$ such that $U\geq_T I^\omega$ and $U$ has the $I$-p.i.p.
\end{enumerate}
\end{theorem}
Taking $I=\fin$ reproduces the difficult part from Dobrinen and Todorcevic's result\footnote{Indeed, as pointed out by Dobrinen and Todorcevic, it is easy to see that $U\cdot U\geq_T\omega^\omega$, so $(1)\Rightarrow(2)$ is straightforward.}. In ~\S\ref{section: above omega to omega}, we first note that in some sense the assumption that $U$ has the $I$-p.i.p in the theorem above is not optimal.
\begin{theorem*}
    Let $U$ be any ultrafilter over a set $X$. If $U$ does not have the $I$-p.i.p then $U\geq_TI^X$. 
\end{theorem*}
In particular, we get the following corollary:
\begin{corollary*}
    If $U$ is not a $p$-point ultrafilter over $\omega$ then $U\geq_T \omega^\omega$.
\end{corollary*}
     This last corollary is of the same flavor as  Theorems \ref{Thm: LouveauVel},\ref{Thm: Todorcevic analytic p}. Hence if we drop the $p$-point assumption, then we also get $U\geq_T\omega^\omega$. Hence, if we are looking for examples for ideals that are not Tukey above $\omega^\omega$ and are not countably generated, we might as well restrict our attention to non-analytic $p$-ideals. This is closely related to \cite[Question 42]{Dobrinen/Todorcevic11}. 
     
     Inside the class of $p$-point, it is known to be consistent that there are $p$-points which are not above $\omega^\omega$ (see \cite{Dobrinen/Todorcevic11}), and the class of rapid p-point is currently the largest subclass of $p$-points known to be above $\omega^\omega$. In the second part of Section~\ref{section: above omega to omega}, we enlarge this class by introducing the notion of $\alpha$-almost-rapid (Definition \ref{Def: almost-rapid}), which is a  weakening of rapidness. 
\begin{theorem*}
    Suppose that $U$ is $\alpha$-almost-rapid, then $U\geq_T \omega^\omega$.
\end{theorem*}

Finally, we prove that the class of almost rapid ultrafilters is a strict extension of the class of rapid ultrafilters, even among $p$-points.
\begin{theorem*}
    Assume CH. Then there is a non-rapid almost-rapid $p$-point ultrafilter.
\end{theorem*} 
This theorem produces a large class of ultrafilters which are above $\omega^\omega$.
This paper is organized as follows:
\begin{itemize}
    \item In~\S \ref{Sec: Sums and prem}, we start with some preliminary definitions and known results. The main goal of this section is the investigation of the cofinal type of $\sum_UV_\alpha$.
    \item  In~\S \ref{section: p.i.p}, we provide a systematic study of the $I$-p.i.p and deterministic ideals.
    \item  In~\S \ref{Sec: proof of the main result}, we prove our results regarding the commutative classes.
    \item In~\S \ref{section: above omega to omega} we investigate the class of ultrafilters Tukey above $\omega^\omega$.
    \item  In~\S \ref{Sec:Question} we present some open problems and possible directions.
 \end{itemize}
\subsection*{Notations} $[X]^{<\lambda}$ denotes the set of all subsets of $X$ of cardinality less than $\lambda$. Let $\fin=[\omega]^{<\omega}$, and $\FIN=\fin\setminus\{\emptyset\}$.
For a collection of sets $(P_i)_{i\in I}$ we let $\prod_{i\in I}P_i=\{f:I\rightarrow \bigcup_{i\in I}P_i\mid \forall i,\, f(i)\in P_i\}$. If $P_i=P$ for every $i$, then $P^I=\prod_{i\in I}P$. Given a set $X\subseteq \omega$, such that $|X|=\alpha\leq\omega$, we denote by $\l X(\beta)\mid \beta<\alpha\r$ be the increasing enumeration of $X$. Given a function $f:A\rightarrow B$, for $X\subseteq A$ we let $f''X=\{f(x)\mid x\in X\}$, for $Y\subseteq B$ we let $f^{-1}Y=\{x\in X\mid f(x)\in Y\}$, and let $\rng(f)=f''A$. Given sets $\{A_i\mid i\in I\}$ we denote by $\biguplus_{i\in I}A_i$ the union of the $A_i$'s when the sets $A_i$ are pairwise disjoint. Two partially ordered set $\mathbb{P},\mathbb{Q}$ are isomorphic, denoted by $\mathbb{P}\simeq \mathbb{Q}$, if there is a bijection $f:\mathbb{P}\rightarrow \mathbb{Q}$ which is order-preserving.

\section{On the Cofinal types of Fubini sums of ultrafilters}\label{Sec: Sums and prem}
\subsection{Some basic definitions and facts}
The principal operation we are considering in this paper is the Fubini/tensor sums and products of ultrafilters.
\begin{definition}\label{Definition: Product and power}
    Suppose that $F$ is a filter over an infinite set $X$ and for each $x\in X$, $G_x$ is a filter over an infinite set $Y_x$. We denote by $\sum_FG_x$ the filter over $\bigcup_{x\in X}\{x\}\times Y_x$, defined by
    $$A\in \sum_FG_x\text{ if and only if }\{x\in X\mid (A)_x\in G_x\}\in F$$
    where $(A)_x=\{y\in Y_x\mid \l x,y\r\in A\}$ is the \textit{$x^{\text{th}}$-fiber} of $A$.  If for every $x$, $G_x=G$ for some fixed $V$ over a set $Y$, then $F\cdot G$ is defined as $\sum_FG$, which is a filter over $X\times Y$. $F^2$ denotes the filter $F\cdot F$ over $X\times X$.
\end{definition}

We distinguish here between $F\cdot G$ and $F\times G$ which is the cartesian product of $F$ and $G$ with the order defined pointwise\footnote{There are papers which consider the filter $\{A\times B\mid A\in F,B\in G\}$ and denote it by $F\times G$, this filter will not be considered in this paper so there is no risk of confusion.}.

A filter $F$ on a regular cardinal $\kappa\geq\omega$ is called \textit{uniform}\footnote{Or non-principal.} if $J^*_{bd}=\{X\subseteq \kappa\mid \kappa\setminus X\text{ is bounded in }\kappa\}\subseteq U$.
\begin{definition} 
Let $F$ be a filter over a 
 regular cardinal $\kappa\geq\omega$.
\begin{enumerate}
    \item $F$ is {\em $\lambda$-complete} if $F$ is closed under intersections of less than $\lambda$ many of its members.
    \item $F$ is {\em Ramsey} if for any function $f:[\kappa]^2\rightarrow 2$ there is an $X\in F$ such that $f\restriction [X]^2$ is constant. 
    \item $F$ is {\em selective} if for every function $f:\kappa \rightarrow \kappa$, there is an $X \in F$ such that
 $f\restriction X$ is  either constant or one-to-one.
\item (Kanamori \cite{KanPPoint}) $F$ is {\em rapid} if for each normal function $f: \kappa \rightarrow \kappa$ (i.e.\ increasing and continuous),
there exists an $X \in F$ such that
$\text{otp}(X \cap f(\alpha)) \leq \alpha$ for each $\alpha< \kappa$.
(i.e.\ bounded pre-images), there is an $X\in F$ such
 that $|f^{-1}(\{\alpha\})\cap X|\le \alpha$ for every $\alpha< \kappa$.
\item $F$ is a {\em $p$-point} if whenever $f:\kappa\rightarrow\kappa$ is unbounded\footnote{Namely, $f^{-1}[\alpha]\notin F$ for every $\alpha<\kappa$} on a set in $F$, it is almost one-to-one mod $F$, i.e. there is an $X\in F$ such that for every $\gamma<\kappa$,  $|f^{-1}[\gamma]\cap X|<\kappa$.
\item $U$ is a {\em $q$-point} 
if every  function $f:\kappa\rightarrow\kappa$  which is almost one-to-one mod $ F $ is injective mod $ F $.

\end{enumerate}

a \textit{$\kappa$-filter} is a  uniform, $\kappa$-complete filter.
\end{definition}
The following facts are well known.
\begin{fact}
    The following are equivalent for a $\kappa$-ultrafilter $U$:
    \begin{enumerate}
        \item $U$ is Ramsey.
        \item $U$ is selective.
        \item $U$ is a $p$-point and a $q$-point.
    \end{enumerate}
\end{fact}
\begin{fact}\label{Fact: Sum is never p-point}
    Suppose that $U,V_\alpha$ are ultrafilters on $\kappa\geq \omega$ where each $V_\alpha$ is uniform. Then $\sum_UV_\alpha$ is not a $p$-point.
\end{fact}
Indeed the function $\pi_1$, the projection to the first coordinate, is never almost one-to-one on a set in $X\in \sum_UV_\alpha$. 
\begin{definition}
    Let $F,G$ be filters on $X,Y$ resp. We say that $F$ is \textit{Rudin-Keisler} below $G$, denoted by $F\leq_{RK} G$, if there is a \textit{Rudin-Keisler} projection $f:Y\rightarrow X$ such that $$f_*(G):=\{A\subseteq X\mid f^{-1}[A]\in G\}=F$$
    We say that are RK-isomorphic, and denote it by $F\equiv_{RK}G$ if there is a bijection $f$ such that $f_*(F)=G$.
\end{definition}
It is well known that if $F\leq_{RK}G\wedge G\leq_{RK}F$ then $F\equiv_{RK}G$ and that $F,G\leq_{RK}F\cdot G$ via the projection to the first and second coordinates respectively. Also, the Rudin-Keisler order implies the Tukey order.
 A Ramsey ultrafilter over $\kappa$ is characterized as being Rudin-Keisler minimal among $\kappa$-ultrafilters.

  Next, let us record some basic terminology and facts regarding cofinal types. Given two directed partially ordered sets $\mathbb{P},\mathbb{Q}$, the Cartesian product $\mathbb{P}\times\mathbb{Q}$ ordered pointwise, is the least upper bound of $\mathbb{P},\mathbb{Q}$ in the Tukey order (see \cite{DobrinenTukeySurvey15}). It follows that $F\times G\leq_T F\cdot G$.  More generally, for partially ordered sets $(\mathbb{P}_i,\leq_i)$ for $i\in I$, we denote by $\prod_{i\in I}(\mathbb{P}_i,\leq_i)$ to be the order over the underlining set $\prod_{i\in I}\mathbb{P}_i$ with the everywhere domination order, namely $f\leq g$ iff for all $i\in I$, $f(i)\leq_i g(i)$. If the order is clear from the context we omit it and just write $\prod_{i\in I}\mathbb{P}_i$. This is the case when $\mathbb{P}_i=U_i$ is a filter ordered by reversed inclusion of an ideal ordered by inclusion. If for every $i\in I$, $\mathbb{P}_i=\mathbb{P}$ we simply write $\mathbb{P}^I$.

\subsection{The cofinal type of sums and products}
The following theorem \cite{Dobrinen/Todorcevic11} provides the starting point for the analyses of the cofinal type of sums of ultrafilters:
\begin{theorem}[Dobrinen-Todorcevic]\label{Thm:DobTod}
    Let $F,G_x$ be filters as in Definition \ref{Definition: Product and power}. Then:
    \begin{enumerate}
        \item $\sum_FG_x\leq_T F\times \prod_{x\in X}G_x$.
        \item If $G_x=G$ for every $x$, then $F\cdot G \leq_T F\times G^X$.
        \item If $F=G$, then $F\cdot F\leq_T F^X$.
    \end{enumerate}
\end{theorem}
These Tukey types are invariant under Rudin-Keisler isomorphic copies of the ultrafilters involved, hence we may assume for the rest of this paper that ultrafilters are defined on regular (infinite) cardinals. 
 It was later discovered \cite{Milovich12} that $(2),(3)$ of Theorem \ref{Thm:DobTod} are in fact an equivalences:
\begin{theorem}[Milovich]\label{Thm: Milovich}
    Let $F,G$ are $\kappa$-filters, then $F\cdot G\equiv_T F\times G^\kappa$ and in particular $F\cdot F\equiv_T F^\kappa$.
\end{theorem}
The proof of Milovich's Theorem go through in case $F$ is any ultrafilter over $\lambda$ and $G$ is $\lambda$-complete.
\begin{corollary}[Milovich]\label{Cor: Milovich corollary}
    For any two $\kappa$-filters $F,G$, $F\cdot (G\cdot G)\equiv_T F\cdot G$ and $F^3\equiv_T F^2$.
\end{corollary}

It is tempting to conjecture that $\sum_UV_\alpha\equiv_T U\times \prod_{\alpha<\lambda}V_\alpha$, however, this will not be the case  in general, as indicated by the following example:
\begin{example}
    Suppose that $U$ and $V$ are Tukey incomparable ultrafilters on $\omega$, and $U\equiv_T U\cdot U$. This situation is obtained for example under  $Cov(\mathcal{M})=\mathfrak{c}$ \footnote{By Ketonen \cite{Ketonen1976}, this assumption implies that there are $(2^{c})^+$-many distinct selective ultrafilters. Then there are two Tukey incomparable selective ultrafilters and by Dobrinen and Todorcevic \cite{Dobrinen/Todorcevic11}, $U\cdot U\equiv_TU$ for any selective ultrafilter.}. The incomparability requirement ensures that $U\times V>_T U$. Let $V_0=V$ and $V_n=U$ for $n>0$. Then $$\sum_U V_n=U\cdot U\equiv_T U<_T U\times V\leq_T U\times \prod_{n<\omega} V_n.$$  
\end{example}
The point of the example is that the sum is insensitive to removing a neglectable set of coordinates, while the product changes if we remove even a single coordinate. Another quite important difference is that $U\times \prod_{x\in X}V_x$ is insensitive to permutations of the indexing set, while $\sum_U V_x$ is.  
Formally, this is expressed by the following fact: \begin{fact}\label{fact: basic fact on turkey of sum}
    Let $U$ be an ultrafilter over $\lambda\geq\omega$ and $U_\alpha$ on $\delta_\alpha$. For every $X\in U$, $\sum_UV_\alpha\leq_T U\times \prod_{\alpha\in X}V_\alpha\leq_T U\times \prod_{\alpha<\lambda} V_\alpha$.
\end{fact}
\begin{proof}
    The right inequality is clear. The left one is also simple, since the set $\mathcal{X}\subseteq \sum_{U}V_\alpha$, of all $Y$ such that $\pi_1''Y\subseteq X$ is a cofinal set in $\sum_{U}V_\alpha$ and therefore the map $F:U\times \prod_{\alpha\in X}V_\alpha\rightarrow \sum_UV_\alpha$ defined by
    $$F(\l Z,\l A_\alpha\mid \alpha\in X\r\r)=\bigcup_{\alpha\in Z\cap X}\{\alpha\}\times A_\alpha$$
    is monotone and has cofinal image.
\end{proof}

In this section, we provide further insight into the cofinal type of $\sum_UV_\alpha$. 
We will focus on $\kappa$-ultrafilters, so our initial assumption is that $U$ is a $\lambda$-ultrafilter for $\lambda\geq\omega$ and $\l V_\alpha\mid \alpha<\lambda\r$ is a sequence of ultrafilters such that each $V_\alpha$ is a $\delta_\alpha$-ultrafilter where $\delta_\alpha\geq\omega$. Towards our first result, consider the set $$\mathcal{B}(U,\l V_\alpha\mid\alpha<\lambda\r)=\{U\times \prod_{\alpha\in X}V_\alpha\mid X\in U\}$$ ordered by the Tukey order. This is clearly a downward-directed set. Our goal is to prove that in some sense, $\sum_UV_\alpha$ is the greatest lower bound of $\mathcal{B}(U,\l V_\alpha\mid\alpha<\lambda\r)$. 
Consider the maps $$\pi_X:U\times \prod_{\alpha\in X}V_\alpha\rightarrow \sum_UV_\alpha, \ \ \ \ \pi_{X,Y}:U\times \prod_{\alpha\in X}V_\alpha\rightarrow U\times \prod_{\alpha\in Y}V_\alpha$$ Defined for $X,Y\in U$ where $Y\subseteq X$ defined by $$\pi_X(\l Z,\l A_\alpha\mid \alpha\in X\r\r)=\bigcup_{\alpha\in X\cap Z}\{\alpha\}\times A_\alpha\ \text{ and }$$ 
$$\pi_{X,Y}(\l Z,\l A_\alpha\mid \alpha\in X\r\r)=\l Z,\l A_\alpha\mid \alpha\in Y\r\r.$$ Then
\begin{enumerate}
    \item $\pi_X$ is monotone cofinal and $\rng(\pi_X)$ is exactly all the sets $B\in \sum_UV_\alpha$ in standard form\footnote{A set $B\in\sum_
UV_\alpha$ is said to be in \textit{standard form} if for every $\alpha<\lambda$, either $(B)_\alpha=\emptyset$ or $(B)_\alpha\in V_\alpha$.} such that $\pi'' B\subseteq X$.
    \item $\pi_{X,Y}$ is monotone cofinal.
    \item $\pi_{Y}\circ \pi_{X,Y}(C)\subseteq \pi_X(C)$.
\end{enumerate}

Suppose that $\sum_UV_\alpha\geq_T \mathbb{P}$. Recall that if $\mathbb{P}$ is complete\footnote{i.e., every bounded subset of $\mathbb{P}$ has a least upper bound.} (e.g. $\mathbb{P}=F$ is a filter ordered by reverse inclusion or any product of complete orders), then $\mathbb{Q}\geq_T\mathbb{P}$ implies that there is a monotone\footnote{$f:\mathbb{Q}\rightarrow\mathbb{P}$ is called monotone if $q_1\leq_{\mathbb{Q}} q_2\Rightarrow f(q_1)\leq_{\mathbb{P}} f(q_2)$.} cofinal map $f:\mathbb{Q}\rightarrow \mathbb{P}$. Suppose that $\mathbb{P}$ is complete and let $g:\sum_UV_\alpha\rightarrow \mathbb{P}$ be monotone cofinal. Define $f_X=g\circ \pi_X$. Then $f_X$ is monotone cofinal from $U\times \prod_{\alpha\in X}V_\alpha$ to $\mathbb{P}$. Moreover, we have that if $Y\subseteq X$ then $$f_Y(\pi_{X,Y}(C))=g(\pi_Y(\pi_{X,Y}(C))\geq_{\mathbb{P}} g(\pi_X(C))=f_X(C)$$
\begin{definition}\label{Def: uniformly below}
A sequence of monotone cofinal maps $$\l f_X:U\times \prod_{\alpha\in X}V_\alpha\rightarrow \mathbb{P}\mid X\in U\r$$ if said to be \textit{coherent} if 
$$(\dag) \ \ \text{whenever } Y\subseteq X, \text{ and }C\in U\times\prod_{\alpha\in X}V_\alpha, \ f_Y(\pi_{X,Y}(C))\geq_{\mathbb{P}} f_X(C).$$
A poset $\mathbb{P}$ is said to be \textit{uniformly below} $\mathcal{B}(U,\l V_\alpha\mid\alpha<\lambda\r)$ if there is a coherent sequence of monotone cofinal maps $\l f_X:U\times \prod_{\alpha\in X}V_\alpha\rightarrow \mathbb{P}\mid X\in U\r$.

\end{definition}
The following theorem says that $\sum_UV_\alpha$ is the greatest lower bound among all the posts uniformly below $\mathcal{B}(U,\l V_\alpha\mid\alpha<\lambda\r)$.
\begin{theorem}\label{thm: uniformly below characterize}
    Suppose that $\mathbb{P}$ is a complete order. Then $\mathbb{P}$ is uniformly below $\mathcal{B}(U,\l V_\alpha\mid\alpha<\lambda\r)$
     if and only if  $\sum_UV_\alpha\geq_T \mathbb{P}$.
\end{theorem}
\begin{proof}
    From right to left was already proven in the paragraph before Definition \ref{Def: uniformly below}. Let us prove from left to right. Let $\l f_X\mid X\in U\r$ be the sequence witnessing that $\mathbb{P}$ is uniformly below $\mathcal{B}(U,\l V_\alpha\mid\alpha<\lambda\r)$. Let $\mathcal{X}\subseteq\sum_{U}V_\alpha$ be the usual cofinal set of all the sets $A\in \sum_UV_\alpha$ is a standard form. Let us define $F:\chi\rightarrow \mathbb{P}$ monotone and cofinal,
    $$F(A)=f_{\pi_1''A}(\l \pi_1''A,\l (A)_\alpha\mid \alpha\in \pi''A\r\r)$$
    We claim first (and most importantly) that $F$ is monotone. Suppose that $A,B\in \chi$ are such that $A\subseteq B$. Then, \begin{enumerate}
        \item [(a.)] $\pi_1''A\subseteq \pi_1'' B$ and
        \item [(b.)] for every $\alpha<\lambda$, $(A)_\alpha\subseteq (B)_\alpha$. 
    \end{enumerate} 
    Define the sequence $\l X_\alpha\mid \alpha\in \pi_1''B\r$ by $X_\alpha=(A)_\alpha$ for $\alpha\in \pi_1''A$ and $X_\alpha=(B)_\alpha$ for $\alpha\in \pi_1''B\setminus\pi_1''A$. Note that $$\pi_{\pi_1''B,\pi_1''A}(\l \pi_1''A,\l X_\alpha\mid \alpha\in \pi_1''B\r\r)=\l \pi_1''A,\l(A)_\alpha\mid \alpha\in \pi_1''A\r\r$$ and that $X_\alpha\subseteq (B)_\alpha$. It follows by monotonicity of the functions, and by $(\dag)$ that
    $$F(A)=f_{\pi_1''A}(\l \pi_1''A,\l(A)_\alpha\mid \alpha\in \pi_1''A\r\r)=f_{\pi_1''A}(\pi_{\pi_1''B,\pi_1''A}(\l \pi_1''A,\l X_\alpha\mid \alpha\in \pi_1''B\r\r)$$
    $$\geq_{\mathbb{P}} f_{\pi_1''B}(\l \pi_1''A,\l X_\alpha\mid \alpha\in \pi_1''B\r\r)\geq_{\mathbb{P}} f_{\pi_1''B}(\l \pi_1''B,\l (B)_\alpha\mid \alpha\in \pi_1''B\r\r=F(B)$$
    To see it is cofinal, let $p\in \mathbb{P}$ be any element, fix any $X\in U$, since $f_X$ is cofinal, there is $\l Z, \l A_\alpha\mid \alpha\in X\r\r\in U\times\prod_{\alpha\in X}V_\alpha$ such that $f_X(\l Z, \l A_\alpha\mid \alpha\in X\r\r)\geq_{\mathbb{P}}p$.
    Consider $A=\cup_{\alpha\in Z\cap X}\{\alpha\}\times A_\alpha$. Then
    $$F(A)=f_{Z\cap X}(\l Z\cap X,\l A_\alpha\mid \alpha\in Z\cap X\r\r)=$$
    $$f_{Z\cap X}(\pi_{X,Z'\cap X}(\l Z\cap X,\l A_\alpha\mid \alpha\in X\r\r))\geq_{\mathbb{P}} f_X(\l Z\cap X,\l A_\alpha\mid \alpha\in X\r)\geq_{\mathbb{P}} p$$
    
\end{proof}



\begin{lemma}\label{Lemma: uniformly below on a base}
    Suppose that $\mathbb{P}$ is complete and for each $X\in U$, $\mathcal{X}_X\subseteq U\times \prod_{\alpha\in X}V_\alpha$ is such that:
    \begin{enumerate}
        \item $\mathcal{X}_X$ is a cofinal subset of $U\times\prod_{\alpha\in X}V_\alpha$.
        \item $f_X:\mathcal{X}_X\rightarrow \mathbb{P}$ is monotone cofinal.
        \item whenever $Y\subseteq X$, $\pi_{X,Y}''\mathcal{X}_X\subseteq \mathcal{X}_Y$ and
        $f_Y(\pi_{X,Y}(C))\geq_{\mathbb{P}}f_X(C)$
    \end{enumerate}
Then $\mathbb{P}$ is uniformly below $\mathcal{B}(U,\l B_\alpha\mid \alpha<\lambda\r)$.
\end{lemma}
\begin{proof}
 Define $f^*_X:U\times\prod_{\alpha\in X}V_\alpha\rightarrow \mathbb{P}$ by $$f^*_X(A)=\sup\{f_X(B)\mid A\subseteq B\in \mathcal{X}_X\}.$$
Note that if $B'\in\mathcal{X}_X$ is such that $B'\subseteq A$, then the set $\{f_X(B)\mid A\subseteq B\in\mathcal{X}_X\}$ is bound in $\mathbb{P}$ by $f_X(B')$ (as $f_X$ is monotone). Hence $f^*_X(A)$ is well defined by completeness of $\mathbb{P}$. It is straightforward that Since $f_X$ is monotone cofinal, $f^*_X$ is monotone cofinal. To see $(\dag)$, suppose that $Y\subseteq X$, and $C\in U\times\prod_{\alpha<\lambda}V_\alpha$, then for every $C\subseteq B\in\mathcal{X}_X$, then by $(3)$ $\pi_{X,Y}(C)\subseteq\pi_{X,Y}(B)\in\mathcal{X}_X$ and $f_{Y}(\pi_{X,Y}(B))\geq_{\mathbb{P}} f_X(B)$. If follows that
$$f^*_X(C)=\sup\{f_X(B)\mid C\subseteq B\in\mathcal{X}_X\}\leq \sup\{ f_Y(\pi_{X,Y}(B))\mid C\leq_{\mathbb{P}} B\in\mathcal{X}_X\}$$
$$\leq_{\mathbb{P}} \sup\{f_Y(B')\mid \pi_{X,Y}(C)\subseteq B'\in\mathcal{X}_Y\}=f^*_Y(\pi_{X,Y}(C)).$$
Hence $\l f_X^*\mid X\in U\r$ is coherent and therefore $\mathbb{P}$ is uniformly below $\mathcal{B}(U,\l V_\alpha\mid\alpha<\lambda\r)$.
\end{proof}
\begin{corollary}\label{cor: below each V_n}
Let $U$ be an ultrafilter on $\lambda\geq\omega$ and that each $V_\alpha$ is a $\delta_\alpha$-complete ultrafilter on some $\delta_\alpha>\alpha$. $\mathbb{P}\leq_T  V_\alpha$ for every $\alpha<\lambda$, then $\mathbb{P}^\lambda\leq_T \sum_UV_\alpha$.
\end{corollary}
\begin{proof}
    We fix for every $\alpha<\lambda$, $f_\alpha:V_\alpha\rightarrow \mathbb{P}$ monotone and cofinal. Now for every $X\in U$, we define a cofinal set $\mathcal{X}_X\subseteq U\times \prod_{\alpha\in X}V_\alpha$ consisting of all the elements $\l Z,\l A_\alpha\mid \alpha\in X\r\r\in U\times \prod_{\alpha\in X}V_\alpha$ such that for every $\alpha<\beta$ in $X$, $f_\alpha(A_\alpha)\leq_{\mathbb{P}} f_\beta(A_\beta)$. \begin{claim}
        $\mathcal{X}_X$ is cofinal in  $U\times \prod_{\alpha\in X}V_\alpha$
    \end{claim}
    \begin{proof}
        Let $\l Z,\l B_\alpha\mid \alpha\in X\r\r$, let us construct $A_\alpha$ recursively. Let $B_0=A_0$. Suppose that $A_\alpha$ for $\alpha\in X\cap \beta$ where defined for some $\beta\in X$. Then for each  $\alpha\in X\cap\beta$, we find (by cofinality of $f_\beta$) a set $C_\alpha\in V_\beta$ such that $f_\alpha(A_\alpha)\leq_{\mathbb{P}}f_{\beta}(C_\alpha)$. Let $A_\beta=B_\beta\cap\bigcap_{\alpha<\beta}C_\alpha$. By $\delta_\beta$-completeness of $V_\beta$, $A_\beta\in V_\beta$. By monotonicity of $f_\beta$, we conclude that for every $\alpha<\beta$, $f_{\alpha}(A_\alpha)\leq_{\mathcal{P}}f_{\beta}(A_\beta)$. It is now clear that $\l Z,\l A_\alpha\mid \alpha\in X\r\r\in \mathcal{X}_X$ and above $\l Z,\l B_\alpha\mid \alpha\in X\r\r$. 
        \end{proof}
          Note that $\pi_{X,Y}''\mathcal{X}_X\subseteq\mathcal{X}_Y$. Define $f_X:\mathcal{X}_X\rightarrow \mathbb{P}^\lambda$ by $$f_X(\l Z,\l A_\alpha\mid \alpha\in X\r\r)=\l f_{X(\alpha)}(A_{X(\alpha)})\mid \alpha<\lambda\r.$$ Let $\l Z,\l A_\alpha\mid \alpha\in X\r\r\in\mathcal{X}_X$, and let $Y\subseteq X$, then $Y(\alpha)\geq X(\alpha)$. Hence,
    by definition of $\mathcal{X}_X$, 
    $f_{X(\alpha)}(A_{X(\alpha)})\leq_{\mathbb{P}} f_{Y(\alpha)}(A_{Y(\alpha)})$. We conclude that
    $$f_X(\l Z,\l A_\alpha\mid \alpha\in X\r\r)=\l f_{X(\alpha)}(A_{X(\alpha)})\mid \alpha<\lambda\r\leq_{\mathbb{P}^\omega}\l f_{Y(\alpha)}(A_{Y(\alpha)})\mid \alpha<\lambda\r$$
    $$=f_Y(\l Z,\l A_\alpha\mid \alpha\in Y\r\r)=f_Y(\pi_{X,Y}(\l Z,\l A_\alpha\mid \alpha\in X\r\r))$$
    Hence by Lemma \ref{Lemma: uniformly below on a base} $\mathbb{P}^\lambda$ is uniformly below $\mathcal{B}(U,\l V_\alpha\mid \alpha<\lambda\r)$ and by Theorem \ref{thm: uniformly below characterize}, $\mathbb{P}^{\lambda}\leq_T\sum_UV_\alpha$.
\end{proof}
In particular, $U_\alpha$ is Tukey-top for a set of $\alpha$'s in $U$, then $\sum_U U_\alpha$ is Tukey top.

It is unclear whether being uniformly below $\mathcal{B}(U,\l V_\alpha\mid \alpha<\lambda\r)$ is equivalent to simply being a Tukey below each $X\in\mathcal{B}(U,\l V_\alpha\mid \alpha<\lambda\r)$. Hence it is unclear if $\sum_UV_\alpha$ is indeed the greatest lower bound of $\mathcal{B}(U,\l V_\alpha\mid \alpha<\lambda\r)$ in the usual sense; if every  $\mathbb{P}$ which is a lower bound in the Tukey order for $\mathcal{B}(U,\l V_\alpha\mid \alpha<\lambda\r)$ is Tukey below $\sum_UV_\alpha$. Let us give a few common configurations of the Tukey relation among the ultrafilters $V_\alpha$ in which $\sum_UV_\alpha$ is the greatest lower bound in the usual sense. Let us denote that by $\sum_UV_\alpha=\inf(\mathcal{B}(U,\l V_\alpha\mid \alpha<\lambda\r))$.

The following is a straightforward corollary from Theorem \ref{thm: uniformly below characterize}:
\begin{corollary}\label{Cor: Concentrating on a single set}
    Let $X_0\in U$, then 
       $\sum_UV_\alpha\equiv_T U\times \prod_{\alpha\in X_0}V_\alpha$ if and only if $U\times\prod_{\alpha\in X_0}V_\alpha$ is uniformly below $\mathcal{B}(U,\l V_\alpha\mid\alpha<\lambda\r)$. In that case $\sum_UV_\alpha=\inf(\mathcal{B}(U,\l V_\alpha\mid \alpha<\lambda\r)$.
\end{corollary}
The second case in which $\sum_UV_\alpha$ turns out to be the greatest lower bound is the following:
\begin{lemma}\label{Lemma: strict infinum configuration}
    Suppose that there is a set $X_0\in U$ such that for every $\alpha< \beta\in X_0$, $V_\alpha$ is a $\kappa$-complete ultrafilter such that $V_\alpha\cdot V_\alpha\equiv_T V_\alpha>_T V_\beta$. Then $\sum_UV_\alpha=\inf(\mathcal{B}(U,\l V_\alpha\mid \alpha<\lambda\r))$ is a strict greatest lower bound. 
\end{lemma}
\begin{proof}
    First note that for every $Y\subseteq X$, by the assumptions, $$V_{\min(Y)}\leq_T\prod_{m\in Y}V_m\leq_T\prod_{m\in Y}V_{\min(Y)}\equiv_T V_{\min(Y)}\cdot V_{\min(Y)}\equiv_T V_{\min(Y)}.$$ Therefore, if $\mathbb{P}\leq_T B$ for every $B\in \mathcal{B}(U,\l V_\alpha\mid \alpha<\lambda\r)$ then $\mathbb{P}\leq_T V_\alpha$ for every $\alpha\in X$. By corollary \ref{cor: below each V_n}, it follows than that $\mathbb{P}\leq_T\sum_U V_\alpha$. Moreover, $\sum_UV_\alpha$ is strictly below $\mathcal{B}(U,\l V_\alpha\mid \alpha<\lambda\r)$, since for every $\beta<\lambda$, $\sum_UV_\alpha\leq_T V_{\beta+1}<_T V_\beta$. 
\end{proof}
We will later show that the assumptions in the above Lemma are consistent. Before that, we consider
the third configuration in which the ultrafilters are increasing, the proof below works only for ultrafilters on $\omega$ and we do not know whether it is possible to generalize it to other ultrafilters. 
\begin{lemma}\label{Lemma: increasing v_n's case}
    Suppose that $U,V_n$ are ultrafilters on $\omega$, such that on a set $X_0\in U$, for every $n\leq m\in X_0$, $V_n\leq_T V_m$. Then $$U\times\prod_{n\in X_0}V_n\equiv_T\sum_UV_n=\inf(\mathcal{B}(U,\l V_\alpha\mid \alpha<\lambda\r))$$
\end{lemma}
\begin{proof} By Corollary \ref{Cor: Concentrating on a single set}, if $\sum_UV_n\equiv_T U\times \prod_{n\in X_0}V_\alpha$, then it must be the greatest lower bound as well. To prove the Tukey-equivalence, first note that $\sum_{U}V_n\leq_T U\times\prod_{n\in X_0} V_n$ by Fact \ref{fact: basic fact on turkey of sum}. For the other direction, define for every $n\in X_0$, $n^+=\min(X_0\setminus n+1)$ and let $f_{n^+,n}:V_{n^+}\rightarrow V_n$ monotone and cofinal. Denote by $n^{+2}=(n^+)^+$ and $n^{+k}=(n^{+(k-1)})^+$ be the $k^{\text{the}}$ successor of $n$ in $X_0$. For any $n<m\in X_0$, suppose that $m=n^{+k}$ and let $f_{m,n}=f_{n^{+k},n^{+(k-1)}}\circ f_{n^{+(k-1)},n^{+(k-1)}}\circ...\circ f_{n^+,n}$. Moreover, let $f_{n,n}:V_n\rightarrow V_n$ be the identity. Hence $f_{m,n}:V_m\rightarrow V_n$ is monotone cofinal, and if $k\in X_0\cap[n,m]$ then $f_{m,n}=f_{k,n}\circ f_{m,k}$.

Let us define a coherent sequence of cofinal maps from a cofinal subset of $U\times \prod_{n\in X}V_n$ to $U\times \prod_{n\in X_0}V_n$ for $X\in U$. Consider the collection $\mathcal{X}_X\subseteq U\times\prod_{n\in X} V_n$ of all $\l Z,\l A_n\mid n\in X\r\r$ such that for all $n,m\in X\cap X_0$, if $n<m$ then $f_{m,n}((A)_m)\subseteq (A)_{n}$. It is straightforward to check that if $Y\subseteq$ then $\pi_{X,Y}''\mathcal{X}_X\subseteq\mathcal{X}_Y$.
    \begin{claim}
        $\mathcal{X}_X$ is cofinal in $U\times \prod_{n\in X}V_n$.
    \end{claim}
    \begin{proof}
        Let $\l Z,\l A_n\mid n\in X\r\r\in U\times \prod_{n\in X} V_n$. We define a sequence $X_n$ be induction on $n\in X$. $X_{n_0}=(A)_{n_0}$. Suppose we have defined $X_{n_k}\in V_{n_k}$ for some $k<m$, For each $k$, we find $C_{m,k}\in V_m$ such that $f_{m,k}(C_{m,k})\subseteq X_{n_k}$. Define $X_{n_m}=(A)_{n_m}\cap(\bigcap_{k<m}C_{m,k})$. By monotonicity of the $f_{m,k}$'s $f_{m,k}(X_{n_m})\subseteq X_{n_k}$. Let $A_1=\l Z,\l X_n\mid n\in X\r\r$, then $A_1\in\mathcal{X}$ and $A_1\geq \l Z,\l A_\alpha\mid n\in X\r\r$.   
    \end{proof} Fix the unique order isomorphism $\sigma_{X,X_0}:X\rightarrow X_0$ (which then satisfy $\sigma(n)\leq n$ as $X\subseteq X_0$) and let $f_X:\mathcal{X}_X\rightarrow U\times\prod_{n\in X_0}V_n$ be defined by $$f_X(\l Z,\l A_n\mid n\in X\r\r)=\l Z,\l f_{\sigma_{X,X_0}^{-1}(n),n}((A)_{\sigma_{X,X_0}^{-1}(n)})\mid n\in X_0\r\r.$$ Clearly, $f_X$ is monotone, let us check that it is cofinal and that the sequence $\l f_X\mid X\in U\restriction X_0\r$ is coherent. Suppose that $C_1=\l Z,\l A_n\mid n\in X_0\r\r\in U\times \prod_{n\in X_0}V_n$. ,
    We find $\l B_n\mid n\in X_0\r\geq \l A_n\mid n\in X_0\r$ such that $f_{m,n}(B_m)\subseteq B_n$ for every $n<m\in X_0$. This is possible as before, constructing the $B_n$'s by induction and the fact that at each step we only have finitely many requirements, so we can intersect the corresponding finitely many sets. Now take $\l Z, \l B_n\mid n\in X\r\r\in\mathcal{X}_X$. Then $$f_X(\l Z,\l B_n\mid n\in X\r\r)=\l Z, \l f_{\sigma_{X,X_0}^{-1}(n),n}(B_{\sigma_{X,X_0}^{-1}(n)})\mid n\in X_0\r\r.$$ Since $\sigma_{X,X_0}^{-1}(n)\geq n$ for every $n\in X_0$, we conclude that $$f_{\sigma_{X,X_0}^{-1}(n),n}(B_{\sigma_{X,X_0}^{-1}(n)})\subseteq B_n,$$ and therefore $f_X$ is cofinal. Similarly, to see $(\dag)$, if $Y\subseteq X\subseteq X_0$, and $\l Z,\l A_n\mid n\in X\r\r\in\mathcal{X}_X$, then $\sigma_{Y,X_0}^{-1}(n)\geq \sigma_{X,X_0}^{-1}(n)$, and therefore, for every $n<\omega$,
    $$f_{\sigma_{Y,X_0}^{-1}(n),n}((A)_{\sigma_{Y,X_0}^{-1}(n)})=f_{\sigma_{X,X_0}^{-1}(n),n}(f_{\sigma_{Y,X_0}^{-1}(n),\sigma_{X,X_0}^{-1}(n)}(A_{\sigma_{Y,X_0}^{-1}(n)}))\subseteq$$
    $$\subseteq f_{\sigma_{X,X_0}^{-1}(n),n}(A_{\sigma_{X,X_0}^{-1}(n)}).$$
    It follows that $f_Y(\pi_{X,Y}(\l Z,\l A_n\mid n\in X\r\r))\geq f_X(\l Z,\l A_n\mid n\in X\r\r)$
\end{proof}
The above Lemma recovers Milovich's theorem \ref{Thm: Milovich}, taking each $V_n=V$ for every $n$.

Our next goal is to prove that the assumptions of Lemma \ref{Lemma: strict infinum configuration} are consistent. This example shows that the cofinal type of $\sum_UV_n$ in general can be quite complicated. To do that, we will need a theorem of Raghavan and Todorcevic from \cite{Raghavan/Todorcevic12} regarding the canonization of cofinal maps from basically generated ultrafilters. The notion of basically generated ultrafilters was introduced by Dobrinen and Todorcevic \cite{Dobrinen/Todorcevic11} as an attempt to approximate the class of ultrafilters which are not Tukey-top. Recall that an ultrafilter $U$ is called \textit{basically generated} if there is a cofinal set $\mathcal{B}\subseteq U$ such that for every sequence $\l b_n\mid n<\omega\r\subseteq \mathcal{B}$ which converges\footnote{A sequence $\l A_n\mid n<\omega\r$ of subsets $\omega$ is said to converge to $A$ if for every $n<\omega$ there is $N<\omega$ such that for every $m\geq N$, $A_m\cap n=A\cap n$.} to an element of $\mathcal{B}$, there is $I\in [\omega]^{\omega}$ such that $\cap_{i\in I}A_i\in U$. A $p$-point ultrafilter $U$ is basically generated as witnessed by $\mathcal{B}=U$ (\cite[Thm. 14]{Dobrinen/Todorcevic11}). Dobrinen and Todorcevic proved that produces and sums of $p$-points must also be basically generated (\cite[Thm. 16]{Dobrinen/Todorcevic11}).
\begin{theorem}[Raghavan-Todorcevic]\label{Thm: Raghavan-Tod}
    Let $U$ be a basically generated ultrafilter and $V$ be any ultrafilter such that $V\leq_T U$. Then there is $P\subseteq \FIN$ such that:
    \begin{enumerate}
        \item $\forall t,s\in P, \  t\subseteq s\Rightarrow t=s$.
        \item $V$ is Rudin-Keisler below $U(P)$, namely, there is $f:P\rightarrow \omega$ such that $V=\{X\subseteq\omega\mid f^{-1}[X]\in U(P)\}$.
        \item $U(P)\equiv_TV$.
    \end{enumerate}
    Where is the filter $U(P)=\{A\subseteq P\mid \exists a\in U. [a]^{<\omega}\subseteq A\}$.
\end{theorem}
The forcing notion $P(\omega)/\fin$ consists of infinite sets, ordered by inclusion up to a finite set. Namely, $ X\leq^* Y$ if $X\setminus Y$ is finite. In the next proposition, we consider the forcing notion $\mathbb{P}=\prod_{n<\omega} P(\omega)/\fin$, where elements of the product have full support. For more information regarding forcing we refer the reader to \cite{Kunen}. 

The following items summarize the properties of $\mathbb{P}$ which we will need:
\begin{itemize}
    \item $\mathbb{P}$ is $\sigma$-closed, and therefore does not add new subsets of $\omega$, and $\omega_1$ is preserved.
    \item For each $n$, the projection $\pi_n$ of $\mathbb{P}$ to the $n^{\text{th}}$ coordinate is a forcing projection from $\mathbb{P}$ to $P(\omega)/\fin$\footnote{A function from $f:\mathbb{P}\rightarrow\mathbb{Q}$ is called a projection of forcing notions if $f$ is order-preserving, $\rng(f)$ is dense in $\mathbb{Q}$, and for every $p\in\mathbb{P}$ and $q\leq_{\mathbb{Q}}p$, there is $p'\leq_{\mathbb{P}}p$ such that $f(p')\leq_{\mathbb{Q}} q$.}.
    \item If $G\subseteq \mathbb{P}$ is generic over $V$, then $U_n:=\overline{\pi_n''G}=\{X\in P(\omega)\mid \exists f\in G. f(n)\leq^* X\}$ is an ultrafilter over $\omega$ in $V[G]$. Moreover, $U_n$ is a generic ultrafilter for $P(\omega)/\fin$.
    \item Each $U_n$ is a selective ultrafilter     \item  $U_n\notin V[\l U_m\mid m\in \omega\setminus\{n\}\r]$.
\end{itemize}
\begin{proposition}\label{prop: example in full support}
    Let $\mathbb{P}$ be a full support product of $\omega$-copies of $P(\omega)/\fin$. Let $G\subseteq \mathbb{P}$ be generic over $V$. Then in $V[G]$ there is a sequence of ultrafilters $V_n$, such that $V_0>_TV_1>_TV_2...$ and $V_n\cdot V_n\equiv V_n$.
\end{proposition}
\begin{proof}
    For each $n<\omega$, $U_n$ is a selective ultrafilter and therefore by Theorem \ref{Thm: Dobrinen and todorcevic rapid p-points}, $U_n\cdot U_n\equiv_T U_n\equiv_T (U_n)^\omega$.
     For every $n<\omega$, define\footnote{We thank Gabe Goldberg for pointing out this definition of $V_n$.} $$V_n=\sum_{U_0}(U_{n+1}\cdot U_{n+2}\cdot...\cdot U_{n+m})_{0<m<\omega}$$
    Note that each $U_{n+1}\cdot...\cdot U_{n+m}$ is basically generated as the product of $p$-points. Therefore, $V_n$ is also basically generated. 
    \begin{lemma}\label{Lemma: Properties of generic mutual}
        \begin{enumerate}
            
            \item $V_n\equiv_T U_0\times \prod_{n<m<\omega}U_m$.
            \item  $V_n\cdot V_n\equiv_T V_n$.
            \item $V_0>_TV_1>_TV_2...$.
            
        \end{enumerate}
    \end{lemma}
    \begin{proof}[\textit{Proof of Lemma.}] For $(1)$, we note that the ultrafilters $U_{n+1}\cdot... \cdot U_{n+m}$ over which we sum in the definition of $V_n$ are increasing in the Tukey order. Hence by Lemma \ref{Lemma: increasing v_n's case}
    $$V_n\equiv_T U_0\times \prod_{0<m<\omega} U_{n+1}\cdot...\cdot U_{n+m}$$
    By Milovich's Theorem \ref{Thm: Milovich}, and by our assumptions, for each $n,m$ $$U_{n+1}\cdot...\cdot U_{n+m}\equiv_T U_{n+1}\times U_{n+2}\cdot U_{n+2}\times...\times U_{n+m}\cdot U_{n+m}\equiv_T U_{n+1}\times...\times U_{n+m}.$$ Hence $$V_n\equiv_T U_0\times \prod_{0<m<\omega} U_{n+1}\times...\times U_{n+m}\equiv_TU_0\times\prod_{0<m<\omega} (U_{n+m})^\omega\equiv_T$$
    $$\equiv_T U_0\times \prod_{0<m<\omega} U_{n+m}\cdot U_{n+m}\equiv_T U_0\times \prod_{0<m<\omega}U_{n+m}$$
        Now for $(2)$, we use $(1)$. For each $n<\omega$
        $$V_n\cdot V_n\equiv_T (V_n)^\omega\equiv_T (U_0\times \prod_{0<m<\omega}U_{n+m})^\omega\equiv_T(U_0)^\omega\times \prod_{0<m<\omega}(U_{n+m})^\omega\equiv_T $$
        $$\equiv_T U_0\times \prod_{0<m<\omega}U_{n+m}\equiv_T V_n$$
        For $(3)$, it follows from $(1)$ that $$V_0\geq_T V_1\geq_T V_2....$$ Suppose toward a contradiction that $V_n\equiv_T V_{n+1}$ for some $n$. Then $U_{n+1}\leq_T V_{n+1}$. Note that $$V_{n+1}\in V[U_0,\l U_m\mid n+1<m<\omega\r].$$ By mutual genericity $U_{n+1}\notin V[U_0,\l U_m\mid n+1<m<\omega\r]$. Since $V_{n+1}$ is basically generated, Theorem \ref{Thm: Raghavan-Tod} implies that there is $P\subseteq FIN$ such that $U_{n+1}\leq_{RK}V_n(P)$. Note that since $\mathbb{P}$ is $\sigma$-closed, $P\in V$ and therefore $V_n(P)\in V[U_0,\l U_m\mid n+1<m<\omega\r]$. Also the Rudin-Keisler projection $f$ such that $f_*(V_n(P))=U_{n+1}$ is in the ground model and therefore $U_{n+1}\in V[U_0,\l U_m\mid n+1<m<\omega\r]$, contradiction. 
    \end{proof}
\end{proof}
It follows that $\sum_{U_0}V_n=\inf(\mathcal{B}(U_0,\l V_n\mid 0<n<\omega\r))$ is a strict greatest lower bound. Let us also prove that $U_0<_T\sum_UV_n$. We will need the following folklore fact.
\begin{fact}
    Suppose that $\sum_UV_n=\sum_UV_n'$ then $\{n<\omega\mid V_n=V'_n\}\in U$
\end{fact}
\begin{proof}
    Just otherwise, $Y=\{n<\omega\mid V_n\neq V'_n\}\in U$, in which case, for every $n\in Y$ take $X_n\in V_n$ such that $X_n^c\in V'_n$. Then $A=\bigcup_{n\in Y}\{n\}\times X_n\in \sum_UV_n$, while $A'=\bigcup_{n\in Y}\{n\}\times X_n^c\in\sum_UV_n'$. However $A\cap A'=\emptyset$ which contradicts $\sum_UV_n=\sum_UV'_n$.
\end{proof}
\begin{proposition}
    $U_0<_T \sum_{U_0}V_n$
\end{proposition}
\begin{proof}
    Otherwise, there would have been a continuous cofinal map $f:U_0\rightarrow \sum_{U_0}
    V_n$. Since $U_0$ is a selective ultrafilter, by Todorcevic \cite{Raghavan/Todorcevic12}, if $V\leq_T U_0$, then there is $\alpha<\omega_1$ such that $V=_{RK} U_0^\alpha$ for some $\alpha<\omega_1$. It follows that $\sum_{U_0}V_n\equiv_{RK}U_0^\alpha$ for some $\alpha<\omega_1$.
    If $\alpha>1$, then $U_0^\alpha=\sum_{U_0}U_0^{\alpha_n}$ for some $\alpha_n<\alpha$ (The $\alpha_n$'s might be constant). 
    It follows that $Y=\{n<\omega\mid V_n=_{RK}U_0^{\alpha_n}\}\in U_0$. Since for any $\beta<\omega_1$, $U_0^\beta\in V[U_0]$, for any $0<n\in Y$, we conclude that $V_n\in V[U_0]$ and in particular $U_1\in V[U_0]$,
    contradicting the mutual genericity. If $\alpha=1$, then $U_0=_{RK}\sum_{U_0}V_n$ which then implies that $\sum_{U_0}V_n$ is a $p$-point, in contradiction to Fact \ref{Fact: Sum is never p-point}. 
    
    
    
\end{proof}
\section{Two properties of filter}\label{section: p.i.p} 
In this section, we present two properties of filters which play a key role in the proof of our main result. The first is the $I$-p.i.p which was introduced in \cite{TomNatasha2}, and the second is a new concept called \textit{ deterministic} ideals. Both of them provide an abstract framework in which one can analyze the connection between the Tukey type of an ultrafilter and ideals related to it. We start this section with a systematic study of the $I$-p.i.p. Many of our results in this section generalize to $\kappa$-filters for $\kappa\geq\omega$. However, we will restrict our attention to ultrafilters on $\omega$, as the main application is the main result regarding the commutativity of Fubini products which is already known for $\kappa$-complete ultrafilters on measurable cardinals \cite{TomNatasha}.
        \subsection{The pseudo intersection property relative to a set.}
        Given set $\mathcal{F}\subseteq P(X)$, we denote by $\mathcal{F}^*=\{X\setminus A\mid A\in \mathcal{F}\}$. When $\mathcal{F}$ is a filter, $\mathcal{F}^*$ is an ideal which we call \textit{the dual ideal}, and when $\mathcal{I}$ is an ideal $\mathcal{I}^*$ is a filter which we call \textit{the dual filter}. Ideals are always considered with the (regular) inclusion order.
        \begin{fact}
            For every filter $F$, $(F,\subseteq)\simeq(F^*,\supseteq)$ and in particular $(F,\subseteq)\equiv_T (F^*,\supseteq)$.
        \end{fact}
        \begin{definition}\label{definition: the I-p.i.p}
            A filter $F$ over a countable set $S$ such that $T\subseteq F^*$ ($T$ is any subset), is said to satisfy the $T$-pseudo intersection property ($T$-p.i.p) if for every sequence $\l X_n\mid n<\omega\r\subseteq F$, there is $X\in F$ such that for every $n$, there is $t\in T$ such that $X\setminus X_n\subseteq t$.
        \end{definition}
         The proof for these simple facts can be found in \cite{TomNatasha2}:
        \begin{fact}
            \begin{enumerate}
                \item Any filter $F$ has the $F^*$-p.i.p.
                \item $F$ is a $p$-point iff $F$ has $\fin$-p.i.p
            \end{enumerate}
        \end{fact}
        The following facts are also easy to verify:
        \begin{fact}
\begin{enumerate}
    \item If $T$ is downward closed with respect to inclusion, then $F$ has the $T$-p.i.p if and only if for every sequence $\l X_n\mid n<\omega\r\subseteq F$, there is $X\in F$ such that $X\setminus X_n\in T$.
    \item $F$ has $\{\emptyset\}$-p.i.p if and only if $F$ is $\sigma$-complete (and therefore, if $F$ is on $\omega$, then it is principal).
\end{enumerate}            
        \end{fact}
       Benhamou 
 and Dobrinen proved the following:
        \begin{proposition}[\cite{TomNatasha2}]\label{Prop: TomNatasha bound with p.i.p} Suppose that $F$ is a filter and $I\subseteq F^*$ is any ideal such that $F$ has the $I$-p.i.p. Then 
        $F^\omega\leq_T F\times I^\omega$.
            
        \end{proposition}
        \begin{theorem}[\cite{TomNatasha2}]\label{Thm: Sufficint condition for product from TomNatash2}
            Suppose that $U$ is an ultrafilter and $I\subseteq U^*$ is an ideal such that:
            \begin{enumerate}
                \item $U$ has the $I$-p.i.p.
                \item $I^\omega\leq_T U$
            \end{enumerate}
            Then $U\cdot U\equiv_T U$.
        \end{theorem}
        This theorem is the important direction in the equivalence of Theorem \ref{Thm: TomNatahsa equivalence for product}.
This subsection is devoted to a systematic study of this property, which will be used in the proof for our main theorem regarding the commutativity of the cofinal types of Fubini products of ultrafilters.
First let us provide an equivalent condition to being $I$-p.i.p, similar to the one we have for $p$-points.
\begin{proposition}\label{Prop: equivalent condition}
    Let $U$ be any ultrafitler. Then the following are equivalent:
    \begin{enumerate}
        \item $U$ has the $I$-p.i.p
        \item For any partition $\l A_n\mid n<\omega\r$ such  that for any $n$, $A_n\notin U$, there is $A\in U$ such that $A\cap A_n\in I$ for every $n<\omega$.
        \item Every function $f:\omega\rightarrow\omega$ which is unbounded modulo $U$ is $I$-to-one modulo $U$, i.e. there is $A\in U$ such that for every $n<\omega$, $f^{-1}[n+1]\cap A\in I$. 
    \end{enumerate}
\end{proposition}
\begin{proof}
The proof is standard and is just a generalization of the usual characterization of $p$-points.
    \begin{enumerate}
        \item [$\underline{(1)\Rightarrow (2)}$] Let $\l A_n\mid n<\omega\r$ be a partition such that $A_n\notin U$. Let $B_n=\omega\setminus A_n\in U$ and by the $I$-p.i.p there is $A\in U$ such that $A\setminus B_n\in I$. It remains to note that $A\setminus B_n=A\cap A_n$ to conclude $(2)$.
        \item [$\underline{(2)\Rightarrow(3)}$] Let 
        $f:\omega\rightarrow\omega$ be unbounded modulo $U$. Let $A_n=f^{-1}[\{n\}]$, then $A_n\notin U$. Apply $(2)$ to the partition $\l A_n\mid n<\omega\r$ to find $A\in U$ such that $A\cap A_n\in I$. For any $n<\omega$, $f^{-1}[n+1]\cap A=\cup_{m\leq n} f^{-1}[\{m\}]\cap A\in I$. Hence $f$ is $I$-to-one modulo $U$.
        \item [$\underline{(3)\Rightarrow(1)}$] Let $\l B_n\mid n<\omega\r\subseteq U$, and let us assume without loss of generality that it is $\subseteq$-decreasing and that $\bigcap_{n<\omega}B_n=\emptyset$. Define $f(n)=\min\{m\mid n\notin B_m\}$. Since $\bigcap_{n<\omega}B_n=\emptyset$, $f:\omega\rightarrow\omega$ is a well defined function.
        Apply $(3)$, to find $A\in U$ such that for every $n<\omega$ $f^{-1}[n+1]\cap A\in I$. Now for each $x\in A\setminus B_n$, $f(x)\leq n$ and therefore $x\in f^{-1}[n+1]\cap A$ and therefore $A\setminus B_n\subseteq f^{-1}[n+1]\cap A\in I$. It follows that $A\setminus B_n\in I$ and that $U$ has the $I$-p.i.p. 
    \end{enumerate}
\end{proof}
        \begin{proposition}
          \begin{enumerate}
       \item If $F$ has $T$-p.i.p and $T\subseteq S$, then $F$ has $S$-p.i.p.
       \item Suppose that $T_1,T_2\subseteq F^*$ are downwards-closed with respect to inclusion, and $F$
 has both the $T_1$-p.i.p and the $T_2$-p.i.p, then $F$ has the $T_1\cap T_2$-p.i.p 
 \item Suppose that $f_*(G)=F$ and $G$ has the $T$-p.i.p then $F$ has the $\{f''t\mid t\in T\}$-p.i.p. 
   \end{enumerate}
  
        \end{proposition}
        \begin{proof}
            For $(1)$, see \cite{TomNatasha2}. For $(2)$, suppose that $U$ has both the $T_1$-p.i.p and the $T_2$-p.i.p. and $T_1,T_2$ are downwards closed.  Let $\l A_n\mid n<\omega\r$ be a sequence, then there are $A,B\in F$ such that for every $n$, $A\setminus A_n\in T_1$ and $B\setminus A_n\in T_2$. It follows that $A\cap B\in F$, fix $n<\omega$, then $A\cap B\setminus A_n$ is included in both  $A\setminus A_n$ and $B\setminus A_n$ which implies that $A\cap B\setminus A_n\in T_1\cap T_1$ as both $T_1,T_2$ are downwards closed. Hence $F$ has the $T_1\cap T_2$-p.i.p.

            For $(3)$, let $\l X_n\mid n<\omega\r\subseteq F$, then $\l f^{-1}[X_n]\mid n<\omega\r\subseteq G$. Therefore, there is $Y\in G$ such that for every $n$ there is $t_n\in T$ such that $Y\setminus f^{-1}[X_n]\subseteq t_n$. 
            Let $X=f''Y\in F$, we have that
            $$X\setminus X_n\subseteq f''[Y\setminus f^{-1}[X_n]]\subseteq f''t_n$$
        \end{proof}
        
      
\begin{corollary}
    Let $F$ be any filter.
Denoted by  
    $U_F\subseteq P(F^*)$ the set generated by all $T$'s such that $F$ has the $T$-p.i.p, namely,
    $$U_{F}=\{S\in P(F^*)\mid \exists T\subseteq S\text{  downwards closed }F\text{ has }T\text{-p.i.p}\}.$$ Then $U_F$ is a filter over $F^*$.
\end{corollary}
$U_F$ is almost an ultrafilter:
\begin{proposition}
    Suppose that $X_1,X_2\subseteq P(F^*)$, and $F$ has the $X_1\cup X_2$-p.i.p, then either $F$ has the $X_1$-p.i.p or $X_2$-p.i.p
\end{proposition}
\begin{proof} Suppose otherwise that $X_1,X_2$ are downward closed, $F$ has the $X_1\cup X_2$-p.i.p , but does not have the neither the $X_1$-p.i.p nor the $X_2$-p.i.p. Then there are sequences $\l A_n\mid n<\omega\r$ and $\l B_n\mid n<\omega\r$ such that for every $A,B$ there are $n_A,m_B$ such that for every $t_1\in X_1$ and every $t_2\in X_2$, $A\setminus A_{n_A}\not\subseteq t_1$ and $B\setminus B_{m_B}\not\subseteq t_2$. Consider the sequence
    $\l \bigcap_{k\leq n}A_k\cap\bigcap_{k\leq n}B_k\mid n<\omega\r$. Then there is $A\in F$ such that for every $l$ there is $t_l\in X_1\cup X_2$ for which $A\setminus \bigcap_{k\leq l}A_k\cap\bigcap_{k\leq l}B_k\subseteq t_l$. For $A$, there are suitable $n_A,m_A$ as above and fix $N=\max(n_A,m_A)$. Without loss of generality, $t_N\in X_1$, in which case, we have $A\setminus A_n\subseteq A\setminus \bigcap_{k\leq N}A_k\cap\bigcap_{k\leq M}B_k\subseteq t_N$, contradicting the choice of $n_A$. 
\end{proof}
   Note that we cannot ensure that either $X_1$ or $X_2$ contains a downward closed subset. Our next result investigates how the $I$-p.i.p is preserved under sums of ideals and ultrafilters.
   
   Let $I$ be an ideal on $X$ and for each $x\in X$ let $J_x$ be ideals on $Y_x$ (resp.). We define the Fubini sum of the ideals $\sum_IJ_x$  over $\bigcup_{x\in X}\{x\}\times Y_x$: For $A\subseteq \bigcup_{x\in X}\{x\}\times Y_x$,
        $$A\in \sum_IJ_x\text{ iff } \{x\in X\mid (A)_x\notin J_x\}\in I.$$
        We denote by $I\otimes J=\sum_IJ$. When $I$ is an ideal on $\omega$, we define transifinitely for $\alpha<\omega_1$ $I^{\otimes\alpha}$. $I^{\otimes 1}=I$, at the successor step $I^{\otimes(\alpha+1)}=I^{\otimes\alpha}\times I$. At limit step $\alpha$, we fix some cofinal sequence $\l \alpha_n\mid n<\omega\r$ unbounded in $\alpha$ and define $I^\alpha=\sum_II^{\otimes\alpha_n}$. The above definition of Fubini sum is nothing but the dual operation of the Fubini sum of filters:
        \begin{fact}\label{Fact: dual operation}
        $(\sum_IJ_x)^*=\sum_{I^*}J_x^*$ and in particular $(I\otimes J)^*=I^*\cdot J^*$.
        \end{fact}
 \begin{proposition}
      Let $F,F_n$ be filters over countable sets. Suppose that $I\subseteq F^*$ and $J_n\subseteq F_n^*$ are ideals for every $n<\omega$. Then if $F$ has $I$-p.i.p and for every $n<\omega$, $F_n$ has $J_n$-p.i.p, then $\sum_FF_n$ has $\sum_IJ_n$-p.i.p. 
    
 \end{proposition}  
\begin{proof}
    Let $\l A_n\mid n<\omega\r$ be a sequence in $\sum_FF_n$. For each $n$, let $X_n=\{m<\omega\mid (A_n)_m\in F_m\}\in F$. We find $X\in F$ such that for every $n<\omega$, $X\setminus X_n\in I$. For each $m\in X$, we consider $E_m=\{n<\omega\mid m\in X_n\}$. If $E_m$ is finite, we let $Y_m=\bigcap_{n\in E_m}(A_n)_m\in F_m$ (if $E_m$ is empty, we let $Y_m=\omega$). Otherwise, we find $Y_m\in F_m$ such that for all $n\in E_m$, $Y_m\setminus (A_n)_m\in J_m$. Let $A=\bigcup_{m\in X}\{m\}\times Y_m$. Then Clearly, $A\in \sum_FF_n$. Let $n<\omega$ and consider $A\setminus A_n$. If $$A\setminus A_n=(\bigcup_{x\in X\cap X_n}\{x\}\times Y_x\setminus (A_n)_x) \cup(\bigcup_{x\in X\setminus X_n}\{x\}\times Y_x).$$ If $x\notin X\setminus X_n$ then $(A\setminus A_n)_x=Y_x\setminus (A_n)_x$. Since $x\in X_n$, we have $n\in E_x$ and therefore $Y_x\setminus (A_n)_x\in J_x$. We conclude that $$\{x<\omega\mid (A\setminus A_n)_x\notin J_n\}=X\setminus X_n\in I.$$ Hence $A\setminus A_n\in \sum_I J_n$.
\end{proof}
One way to obtain non-trivial sets $T$ for which an ultrafilter $U$ has the $T$-p.i.p is by intersecting $U$ with another ultrafilter: 
 \begin{theorem}\label{Thm: intersection of ultrafilters and p.i.p}
     Suppose that $U_1,U_2,...U_n$ are any ultrafilters, then for each $1\leq i\leq n$, $U_i$ have the $(U_1\cap U_2\cap...\cap U_n)^*$-p.i.p. In particular $U_i^\omega\leq_T U_i\times (U_1\cap U_2\cap...\cap U_n)^{*\omega}$.
 \end{theorem}
 \begin{proof}
     Fix any $1\leq i\leq n$. Note that $(\bigcap_{j=1}^n U_j)^*=\bigcap_{j=1}^n (U_j)^*$. Suppose otherwise, then there is a sequence $\l X_n\mid n<\omega\r\subseteq U_i$ such that for every $X\in U_i$ there is $n<\omega$ such that $X\setminus X_n\notin \bigcap_{j=1}^n U_j^*$. Since $X\setminus X_n\in U_i^*$, this means that there is $j\neq i$ such that $X\setminus X_n\notin U_j^*$. Since $U_j$ is an ultrafilter, it follows that $X\setminus X_n\in U_j$ and therefore $X\in U_j$. We conclude that $U_i\subseteq \bigcup_{j\neq i}U_j$. There must be $j\neq i$ such that $U_i\subseteq U_j$, just otherwise, for each $j\neq i$ find $X_j\in U_i$ such that $X_j\notin U_j$. Then $X^*=\bigcap_{j\neq i}X_j\in U_i$. But then there is $j'\neq i$ such that $X^*\in U_{j'}$. It follows that $X_{j'}\in U_{j'}$, contradicting our choice of $X_{j'}$. Since ultrafilters are maximal with respect to inclusion, and $U_i\subseteq U_j$, we conclude that $U_i=U_j$. Now this is again a contradiction since for some (any) $X$, $X_n$ and $X\setminus X_n$ disjoint and both in $U$. 
 \end{proof}
 The argument above can be generalized to an infinite set of ultrafilters in some cases. 
 Recall that $U$ is an accumulation point (in the topological space $\beta\omega\setminus\omega$) of a set of ultrafilters $\mathcal{A}\subseteq \beta\omega\setminus\omega$ if and only if $U\subseteq \bigcup \mathcal{A}\setminus\{U\}$.
 \begin{proposition}
     Suppose that $U$ is not an accumulation point of $\mathcal{A}\subseteq \beta \omega\setminus\omega$. Then $U$ has the $(\bigcap\mathcal{A})^*$-p.i.p. In particular $U^\omega\leq_T U\times (\bigcap\mathcal{A})^\omega$.
 \end{proposition}
 \begin{proof}
     Otherwise, we get that for some sequence $\l X_n\mid n<\omega\r\subseteq U$, for every $X\in U$, there is $n$ such that $X\setminus X_n\in V$ for some $V\in \mathcal{A}$. Since $X\setminus X_n\notin U$, it follows $V\neq U$. It follows that $U\subseteq \bigcup\mathcal{A}\setminus\{U\}$, contradiction. 
 \end{proof}
 
Recall that a sequence $\l U_n\mid n<\omega\r$ of ultrafilters on $\omega$ is called \textit{discrete} if there are disjoint sets $A_n\in U_n$. This is just equivalent to being a discrete set in the space $\beta\omega\setminus\omega$. In particular, no point $U_n$ is in the closure of the others.
 
 \begin{corollary}
     Suppose that $U$ is not in the closure of the $U_n$'s, namely, there is a set $A\in U$ such that for every $n$, $A\notin U_n$, then $U$ has the $\bigcap_{n<\omega}U_n^*$-p.i.p. 
 \end{corollary}
 \begin{corollary}
     If $U_n$ is discrete then each $U_n$ has the $\bigcap_{n<\omega}U_n^*$-p.i.p. 
 \end{corollary}
 The partition given by the discretizing sets of a discrete sequence of ultrafilters can be used to compute the cofinal type of the filter obtained by intersecting the sequence. 
 \begin{proposition}
     Suppose that $U_n$ is discrete. Then  $\bigcap_{n<\omega}U_n\equiv_T \prod_{n<\omega} U_n$.
 \end{proposition}
 \begin{proof}
     Let $A_n$ be a partition of $\omega$ so that $A_n\in U_n$. Define $f(\l B_n\mid n<\omega\r)=\bigcup_{n<\omega} B_n\cap A_n$. It is clearly monotone. If $X\in\bigcap_{n<\omega} U_n$, we let $X\cap A_n=B_n\in U_n$ Then $X=X\cap\omega=X\cap(\bigcup_{n<\omega}A_n)=\bigcup_{n<\omega}X\cap A_n=f(\l B_n\mid n<\omega\r)$. To see it is unbounded, suppose that $f''\mathcal{A}$ is bounded by $B\in \bigcap_{n<\omega} U_n$, then for every $A\in\mathcal{A}$, $A\cap A_n\supseteq B\cap A_n$. and therefore $\l B\cap A_n\mid <\omega\r$ would bound $\mathcal{A}$.
 \end{proof}
 \begin{corollary}\label{Cor: increasing intersection of discrete}
     If $\l U_n\mid n<\omega\r$ is a discrete sequence of ultrafilters then for every $X\subseteq Y$, $\bigcap_{n\in X}U_n\leq_T \bigcap_{m\in Y}U_m$
 \end{corollary}
 
 \subsection{Simple and deterministic ideals}
 \begin{definition}
     An ideal $I$ is \textit{simple} if for every ideal $J$, $I\subseteq J$, $I\leq_T J$.
 \end{definition}
 Clearly any ultrafilter is simple, also $\fin$ is simple as $\fin\equiv_T\omega\leq_T F$ for any filter $F$ which is non-principal. To construct examples of ideals which are not simple, we have the following lemma.
 \begin{lemma}\label{Lemma: non-sipmle discrete}
     Suppose that $\l U_n\mid n<\omega\r$ is a sequence of discrete ultrafilters, $X\subseteq \omega$ and $n<\omega$  such that $U_n\not\leq_T \prod_{m\neq n}U_m$, then for every $X\cup\{n\}\subseteq Y$, 
     $\bigcap_{n\in X}U_n<_T \bigcap_{m\in Y}U_m$
 \end{lemma}
 \begin{proof}
     $\bigcap_{n\in X}U_n\leq_T \bigcap_{m\in Y}U_m$ follows from Corollary \ref{Cor: increasing intersection of discrete}. To see that it is strict, suppose otherwise, and let $t\in Y\setminus X$, then $U_t\leq_T \prod_{m\in Y} U_m$ but by assumption $U_t\not\leq_T \prod_{n\in X}$ and therefore $\bigcap_{n\in X}U_n\not\equiv_T \bigcap_{m\in Y}U_m$. 
 \end{proof}
 
 Clearly, taking any $X\cup\{n\}\subseteq Y$ in the previous lemma, we get $\bigcap_{m\in Y}U_m\subseteq\bigcap_{n\in X}U_n$. Hence $\bigcap_{m\in Y}U_m$ is not simple.
\begin{remark}
    A sequence $\l U_n\mid n<\omega\r$ of $\omega$-many mutually generic ultrafilters for $P(\omega)/\fin$  would be a discrete sequence satisfying the assumptions of Lemma \ref{Lemma: non-sipmle discrete}. To see this, note that by Lemma \ref{Lemma: Properties of generic mutual} $U_0\times\prod_{1<m<\omega}U_n$ is Tukey equivalent to some basically generated ultrafilter in $V[U_0,\l U_m\mid m>1\r]$. Now we argue as in Lemma \ref{Lemma: Properties of generic mutual}, concluding that $U_1\not\leq_T U_0\times \prod_{1<m<\omega}U_m$. 
\end{remark}
\begin{definition}\label{Def: deterministinc}
            We say that an ideal $I$ is \textit{deterministic} if there is a cofinal set $\mathcal{B}\subseteq I$ such that for every $\mathcal{A}\subseteq \mathcal{B}$, $\bigcup\mathcal{A}\in I$ or $\bigcup\mathcal{A}\in I^*$. 
        \end{definition}
         \begin{example}
            We claim that $\fin$ is deterministic. Indeed, let $\mathcal{B}=\omega$. Then clearly, $\mathcal{B}$ is a cofinal in $\fin$. Suppose that $\mathcal{A}\subseteq\omega$ is such that $\bigcup\mathcal{A}\notin fin$, then $\mathcal{A}$ is an unbounded set of natural numbers and therefore $\bigcup\mathcal{A}=\omega\in \fin^*$. 
        We will see later that
            for every $\alpha$,  $\fin^{\otimes\alpha}$ is deterministic.
        \end{example}
        The reason that deterministic ideals are interesting is due to the following proposition:
        \begin{proposition}
            If $I$ is deterministic then $I$ is simple.
        \end{proposition}
        \begin{proof}
            Let $I\subseteq J$ and let $\mathcal{B}\subseteq I$ be the cofinal set witnessing that $I$ is deterministic. Let us prove that the identity function $id:\mathcal{B}\rightarrow J$ is unbounded. Suppose that $\mathcal{A}\subseteq\mathcal{B}$ is unbounded, then $\bigcup\mathcal{A}\notin I$, since otherwise, as $\mathcal{B}$ is cofinal in $I$, there would have been $b\in \mathcal{B}$ bounding $\mathcal{A}$. By definition of deterministic ideals, it follows that $\bigcup\mathcal{A}\in I^*$, and since $I\subseteq J$, $I^*\subseteq J^*$ hence $\bigcup\mathcal{A}\in J^*$. We conclude that $\bigcup\mathcal{A}\notin J$, namely, $\mathcal{A}$ is unbounded in $J$. Hence the identity function witnesses that $I\equiv_T \mathcal{B}\leq_T J$.
        \end{proof}
        \begin{proposition}\label{Prop: deterministic invariates}
            Suppose that $I\subseteq X$ is a deterministic ideal over $X$.
            \begin{enumerate}
                \item If $\pi:X\rightarrow Y$ is injective on a set in $I$. Then  $$\pi_*(I):=\{a\mid \pi^{-1}[a]\in I\}$$ is deterministic. 
                \item If $A\subseteq X$, then $I\cap P(A)$ is deterministic.
                \end{enumerate}
            \end{proposition}
            \begin{proof}
 For $(1)$, let $\mathcal{B}\subseteq I$ be a witnessing cofinal set. Let $\mathcal{C}=\{(Y\setminus(\pi''[X\setminus b]))\mid b\in\mathcal{B}\}$. Then $\mathcal{C}$ is a cofinal set in $\pi_*(I)$.
 Indeed, if $A\in \pi_*(I)$, then $\pi^{-1}[A]\in I$ and there is $b\in \mathcal{B}$ such that $b\supseteq \pi^{-1}[A]$. If $y\notin Y\setminus( \pi''X\setminus b)$, then $y=\pi(x)$ for some $x\in X\setminus b$.
 Since $\pi^{-1}[A]\subseteq b$, then $x\notin \pi^{-1}[A]$ which then implies that $y=\pi(x)\notin A$. We conclude that $A\subseteq Y\setminus (\pi''[X\setminus b])\in\mathcal{C}$, as wanted. We claim that $\mathcal{C}$ witnesses that $\pi_*(I)$ is deterministic.
 Let $\mathcal{A}\subseteq\mathcal{B}$ be such that $\bigcup_{a\in \mathcal{A}}Y\setminus(\pi''[X\setminus a])\notin \pi_*(I)$.
 Then $\pi^{-1}[\bigcup_{a\in \mathcal{A}}Y\setminus(\pi''[X\setminus a])]\notin I$. 
 Simplifying the above set we have
 $$\pi^{-1}[\bigcup_{a\in \mathcal{A}}Y\setminus(\pi''[X\setminus a])]=\pi^{-1}[Y\setminus(\bigcap_{a\in \mathcal{A}}(\pi''[X\setminus a]))]=X\setminus \bigcap_{a\in \mathcal{A}} \pi^{-1}[\pi''[X\setminus a]]$$
$$=\bigcup_{a\in\mathcal{A}}X\setminus \pi^{-1}[\pi''[X\setminus a]]= \bigcup_{a\in\mathcal{A}} a$$
The last inclusion holds as for each $a$, $X\setminus a=\pi^{-1}[\pi''[X\setminus a]]$ as $\pi$ is one-to-one. Then $a=X\setminus(X\setminus a)=X\setminus \pi^{-1}[\pi''[X\setminus a]]$. It follows that $\bigcup_{a\in\mathcal{A}} a\notin I$. Since $\mathcal{A}\subseteq \mathcal{B}$, we conclude that $$\pi^{-1}[\bigcup_{a\in \mathcal{A}}Y\setminus(\pi''[X\setminus a])]=\bigcup_{a\in\mathcal{A}} a\in I^*.$$ Namely, $\bigcup_{a\in \mathcal{A}}Y\setminus(\pi''[X\setminus a])\in \pi_*(I^*)=\pi_*(I)^*$. 



For $(2)$, again, let $\mathcal{B}\subseteq I$ be a cofinal set witnessing that $I$ is deterministic.  Consider $\mathcal{C}=\{b\cap A\mid b\in\mathcal{B}\}$. Then $\mathcal{C}$ is cofinal in $I\cap P(A)$. If $\bigcup_{b\in\mathcal{A}}b\cap A\notin I\cap P(A)$, then $\bigcup_{b\in\mathcal{A}}b\cap A\notin I$ (as it is clearly in $P(A)$). It follows that $\bigcup_{b\in\mathcal{A}}b\notin I$ and since $I$ is deterministic, $\bigcup_{b\in\mathcal{A}}b\in I^*$. 

It follows that $ \bigcap_{b\in\mathcal{A}}X\setminus b\in I$ and  $A\cap \bigcap_{b\in\mathcal{A}}X\setminus b\in I\cap P(A)$.    But $$A\cap \bigcap_{b\in\mathcal{A}}X\setminus b=\bigcap_{b\in\mathcal{A}}A\setminus (b\cap A)=A\setminus \bigcup_{b\in\mathcal{A}}b\cap A.$$ Hence  $\bigcup_{b\in\mathcal{A}}b\cap A\in (I\cap P(A))^*$.
            \end{proof}
            Note that $(2)$ above can be vacuous if $A\in I$, since in that case $I\cap P(A)$ is not proper. So we should at least assume that $A\in I^+$. Generally speaking, it is unclear whether an ideal relative to a positive set has the same Tukey-type. However, if the ideal is deterministic, this type does not change:
            \begin{fact}\label{Fact: Positive deterministic ideal is Tukey equivalent} Suppose that $I$ is an ideal over $X$ and $A\in I^+$, then $I\equiv_T I\cap P(A)$. 
            \end{fact}
            \begin{proof}
            Let $\mathcal{B}$ be a witnessing cofinal set for $I$, and let $f:\mathcal{B}\rightarrow I\cap P(A)$ be the map $f(b)=b\cap A$. Then
                clearly, the map is monotone and its image is the cofinal set $\mathcal{C}$ from the proof of $(2)$ from the previous proposition (and therefore cofinal). To see it is unbounded, suppose that $\bigcup\mathcal{A}\notin I$, then $\bigcup\mathcal{A}\in I^*$ and the computation from the previous proposition applies to show that $\bigcup f''\mathcal{A}\in (I\cap P(A))^*$ and in particular not in $I$. Hence $f$ is unbounded.
            \end{proof}
            
        \begin{corollary}\label{Cor: sufficient condition for fubini power}
            Suppose that $I^\omega\equiv_TI$, $I$ is deterministic. Then for every ultrafilter $U$ such that $I\subseteq U^*$ and $U$ satisfies the $I$-p.i.p, $U\cdot U\equiv_T U$. 
        \end{corollary}
        \begin{proof}
            Since $I$ is deterministic, $I^\omega\equiv_T I\leq_T U$. Since $U$ has the $I$-p.i.p we can apply Theorem \ref{Thm: Sufficint condition for product from TomNatash2} to conclude that $U\cdot U\equiv_T U$.
        \end{proof}
        Given any $\{X_i\mid i\in N\}$, where each $X_i\subseteq P(\omega)$, there is the smallest ideal (might not be proper) that contains all the $X_i$. We denote this ideal by $I(\{X_i\mid i\in N\})$. It is generated by the sets $\{\bigcup_{i\in M}b_i\mid b_i\in X_i , \ M\in[N]^{<\omega}\}$. We can replace each $X_i$ be some cofinal set in $B_i$ and obtain the same generated ideal.
\begin{theorem} Suppose that $I$ is an ideal $I=I(\l I_n\mid n<\omega\r)$, where $I_n\subseteq I$ are deterministic ideals. Then $I$ is deterministic. 
             \end{theorem}
             \begin{proof}
                     
                 Let $\mathcal{B}_n\subseteq I_{n}$ be a cofinal set witnessing $I_{n}$ being deterministic. Let $\mathcal{B}'_n=\{b\cup n\mid b\in \mathcal{B}_n\}$. Then $I=I(\l \mathcal{B}_n\mid n<\omega\r)$. 
                 Consider the cofinal set $\mathcal{B}$ of all sets of the form $X_{T,(b'_i)_{i\in T}}:=\bigcup_{n\in T}b'_n$ where $b'_n\in \mathcal{B}'_n$ and $T\in[\omega]^{<\omega}$. 
                 Then $\mathcal{B}$ is a cofinal set in $I$.  Suppose $\bigcup_{j\in S}X_{T_j,(b^j_i)_{i\in T_j}}\notin I$. Note that $X_{T,(b'_i)_{i\in T}}\supseteq \max(T)$. And so, if $\bigcup_{j\in S}T_j$ is unbounded in $\omega$, then $\bigcup_{j\in S}X_{T_j,(b^j_i)_{i\in T_j}}=\omega\in I^*$. Otherwise, $\bigcup_{j\in S}T_j$ is bounded by some $N$ and therefore $$\bigcup_{j\in S}X_{T_j,(b^j_i)_{i\in T_j}}\subseteq \{0,...,N\}\cup \bigcup_{i\leq N}(\bigcup_{b\in \mathcal{A}_i}b),$$ where $\mathcal{A}_i=\{b^i_j\mid i\in T_j\}\subseteq \mathcal{B}_i$. Since this is a finite union of sets which is not in $I$, there is $i\leq N$ such that $\bigcup_{b\in \mathcal{A}_i}b\notin I$, and in particular, $\bigcup_{b\in \mathcal{A}_i}b\notin I_i$. Since $\mathcal{B}_i$ is a witness for $I_i$ being deterministic, $\bigcup_{b\in \mathcal{A}_i}b\in I_i^*\subseteq  I^*$. Since $\bigcup_{b\in \mathcal{A}_i}b\subseteq \bigcup_{j\in S}X_{T_j,(b^j_i)_{i\in T_j}}$, it follows that $\bigcup_{j\in S}X_{T_j,(b^j_i)_{i\in T_j}}\in I^*$.
             \end{proof}
\begin{proposition}\label{Prop: product of deterministic is also}        Suppose that $\fin\subseteq I$ be a deterministic ideal over $\omega$, and $\l J_n\mid n<\omega\r$ is a sequence of deterministic ideals over $\omega$ such that for every $n<\omega$, $J_{n+1}\geq_T J_n$. Then $\sum_I J_n$ is deterministic.
        \end{proposition}
\begin{proof}
Let $\mathcal{B}_n\subseteq J_n$ witness that $J_n$ is deterministic, and $\mathcal{D}$ witness that $I$ is deterministic. Let $\l f_{m,n}:J_n\rightarrow J_m\mid n\leq m<\omega\r$ be a sequence of unbounded maps. Denote by $\mathcal{B}$ the set of all  sequences $\vec{b}=\l b_n\mid n<\omega\r\in \prod_{n<\omega}\mathcal{B}_n$ such that for every $n<m<\omega$, $b_{m}\supseteq f_{m,n}(b_{n})$. For $\vec{b}\in\mathcal{B}$  and $A\in \mathcal{D}$ defined $$C_{A,\vec{b}}=\bigcup_{n\in A}\{n\}\times \omega\cup\bigcup_{n\notin A}\{n\}\times b_n.$$
Let us show that $\mathcal{C}=\{C_{A,\vec{b}}\mid A\in \mathcal{D}, \vec{b}\in \mathcal{B}\}$ is  cofinal in $\sum_IJ_n$. Let $Z\in \sum_I J_n$, and $A=\{n\mid (Z)_n\notin J_n\}$. By definition of $\sum_I J_n$,  $A\in I$, so there is $D\in\mathcal{D}$ such that $A\subseteq D$. 
Construct an increasing sequence $\langle b_n\mid n<\omega\r\in \prod_{n<\omega}\mathcal{B}_n$ such that  for every $n\notin A$, $(Z)_n\subseteq b_n$ and for every $n<m<\omega$, $f_{m,n}(b_n)\subseteq b_m$. It is possible to construct such a sequence recursively.  At stage $n$, we need to find $b_n$ such that for all $k<n$, $f_{n,k}(b_k)\subseteq b_n$ and if $n\notin A$, then also $(Z)_n\subseteq b_n$. These are finitely many sets in $J_n$ and therefore, since $\mathcal{B}_n$ is a cofinal set in $J_n$, we can find a single $b_n\in\mathcal{B}_n$ including the union of these sets. It follows that for every $n$, $(Z)_n\subseteq (C_{D,{\vec{b}}})_n$ and therefore $Z\subseteq C_{D,\vec{b}}$. Let us prove that $\mathcal{C}$ witnesses that $\sum_I J_n$ is deterministic. Suppose that $\bigcup_{i\in T}C_{X_i,\vec{b}^i}\notin \sum_I J_n$. Then $A:=\{n<\omega\mid (\bigcup_{i\in T}C_{X_i,\vec{b}^i})_n\notin J_n\}\notin I$. 
    Note that $(\bigcup_{i\in T}C_{X_i,\vec{b}^i})_n=\bigcup_{i\in T}(C_{X_i,\vec{b}^i})_n$. 
    Let us split into cases: If $A\subseteq \bigcup_{i\in T}X_i$, then $\bigcup_{i\in T}X_i\notin I$ and since $\{X_i\mid i\in T\}\subseteq\mathcal{D}$, $\bigcup_{i\in T}X_i\in I^*$. Now for every $n\in \bigcup_{i\in T}X_i$, there is $i_0\in T$ such that  $(C_{X_{i_0},\vec{b}^{i_0}})_n=\omega$ and in particular $\omega=\bigcup_{i\in T}(C_{X_i,\vec{b}^i})_n\in J_n^*$. We conclude that $\{n<\omega\mid (\bigcup_{i\in T}C_{X_i,\vec{b}^i})_n\in J_n^*\}\supseteq \bigcup_{i\in T}X_i\in I^*$. Hence $\bigcup_{i\in T}C_i\in (\sum_{I}J_n)^*$.
    
    Otherwise, consider $n_0\in A$ such that for every $i\in T$, $n_0\notin X_i$. Then $(C_{X_i,\vec{b}^i})_{n_0}\in \mathcal{B}_{n_0}$ for every $i\in T$ and $\bigcup_{i\in T}(C_{X_i,\vec{b}^i})_{n_0}\notin J_{n_0}$. Since $\mathcal{B}_{n_0}$ witnesses that $J_{n_0}$ is deterministic, $\bigcup_{i\in T}(C_{X_i,\vec{b}^i})_{n_0}\in J^*_{n_0}$. For every $n\geq n_0$, either there is $i\in T$ such that $n\in X_i$, and as we have seen, $\cup_{i\in T}(C_{X_i,\vec{b}^i})_n=\omega\in J^*_n$. Otherwise,  for every $i\in T$, $(C_{X_i,\vec{b}^i})_n=b^i_n\in \mathcal{B}_n$, and by the assumption,
    $f_{n,n_0}(b^i_{n_0})\subseteq b^i_{n}$. Since  $\bigcup_{i\in T}b^i_{n_0}\notin J_{n_0}$, and $f_{n,n_0}$ is unbounded, $\bigcup_{i\in T}f_{n,n_0}((C_{X_i,\vec{b}^i})_n)\notin J_n$. Since $\mathcal{B}_n$ witnesses that $J_n$ is deterministic, it follows that $\bigcup_{i\in T}f_{n,n_0}((C_{X_i,\vec{b}^i})_n)\in J^*_n$. We conclude that for every $n\geq n_0$, $\bigcup_{i\in T}(C_{X_i,\vec{b}^i})_n\in J^*_n$. Since $\{n_0,n_0+1,...\}\in I^*$, we have that $\bigcup_{i\in T}C_{X_i,\vec{b}^i}\in (\sum_I J_n)^*$ as wanted.

\end{proof}
\begin{corollary}\label{cor: product of deterministic is also}
            Suppose that $I,J$ are deterministic ideals over $\omega$. Then $I\otimes J$ is deterministic.
\end{corollary}
        Recall that for an ideal $I$ on $\omega$ we define recursively the ideal $I^{\otimes\alpha}$ on $\omega^\alpha$ by setting $I^1=I$,   $I^{\otimes(\alpha+1)}=I\otimes I^{\otimes\alpha}$ on $\omega^{\alpha+1}=\omega\times \omega^\alpha$. For limit $\alpha<\omega_1$, fix a cofinal sequence $\l\alpha_n\mid n<\omega\r$ in $\alpha$, then $I^{\otimes\alpha}=\sum_II^{\otimes\alpha_n}$ is an ideal on $\bigcup_{n<\omega}\{n\}\times \omega^{\alpha_n}$. It is clear that $\alpha\leq\beta$ then $I_\alpha\leq_{RK}I_\beta$ (see for example \cite[Lemma 3.2]{TomNatasha2}). \begin{corollary}\label{Cor: I times fin is always deterministic}
                 If $I$ is deterministic then for every $\alpha<\omega_1$, $I^{\otimes\alpha}$ is deterministic. In particular $\fin^{\otimes\alpha}$ is deterministic for every every $\alpha<\omega_1$. \end{corollary}
                 \begin{proof}
                     By induction on $\alpha$. For $\alpha=1$ $I$ is deterministic by assumption. Suppose that this was true for $\alpha$, then $I^{\alpha+1}=I\otimes I^{\otimes \alpha}$ which is deterministic from the previous corollary. Similarily, at limit steps we just use the induction hypothesis and Proposition \ref{Prop: product of deterministic is also}.

                 \end{proof}
        In \cite{TomNatasha2} it was noticed that a generic ultrafilter for $P(X)/I$ where $I$ is a $\sigma$-ideal satisfies the $I$-p.i.p. now together with Corollary \ref{Cor: sufficient condition for fubini power}, we recover the result from \cite{TomNatasha2} in our abstract settings.
        \begin{fact}
            For every $\alpha\geq 2$, $I^{\otimes\alpha}\equiv_T I\otimes I\equiv_T I^\omega$.
        \end{fact}
        \begin{proof}
            By induction on $\alpha\geq 2$. For $\alpha=2$, $$(I\otimes I)^\omega\equiv_T (I^\omega)^\omega\simeq I^\omega\equiv_T I\otimes I.$$
            For successor $\alpha+1$, we use that fact that $I\otimes I\leq_T I^{\otimes\alpha}$ to see that  $$I^{\otimes(\alpha+1)}\equiv_T I\otimes I^{\otimes \alpha}\equiv_T I\otimes (I\otimes I)\equiv_T I\otimes I$$
            At limit steps we have that $$I^{\otimes\alpha}\equiv_T \sum_{I}I^{\otimes\alpha}\equiv_T\sum_{I}I\times I\equiv_T I\otimes( I\otimes I)\equiv_T I\otimes I$$
        \end{proof}
        In particular, $(I^{\otimes \alpha})^{\omega}\equiv_T I^{\otimes \alpha}$.
         \begin{corollary}
        Let $I$ be a deterministic $\sigma$-ideal. Then for every $2\leq\alpha<\omega_1$, a generic ultrafilter $G$ for $P(\omega^\alpha)/I^{\otimes\alpha}$ satisfies $G\cdot G\equiv_T G$.
    \end{corollary}
    \begin{proof}
        $G$ has the $I^{\otimes\alpha}$-p.i.p, and  by Corollary \ref{Cor: I times fin is always deterministic} $I^{\otimes\alpha}$ is deterministic. Also by the previous fact,  $(I^{\otimes\alpha})^\omega\equiv_T I^{\otimes\alpha}$ . Thus by Corollary \ref{Cor: sufficient condition for fubini power} $G\cdot G\equiv_T G$.
    \end{proof}
For $\alpha=1$ we can also say something:
\begin{proposition}\label{Proposition Generic via simple ideal}
    Let $I$ be a simple $\sigma$-ideal. If $G\subseteq P(\omega)/I$ is $V$-generic then $G\cdot G\equiv_T G\times I^\omega$. So in particular if $I^\omega\equiv_T I$, $G\cdot G\equiv_T G$
\end{proposition}
\begin{proof}
    Since $G$ has the $I$-p.i.p, $G\cdot G\leq_T G\times I^\omega$. In the other direction, we known that $I^*\subseteq G$ and since $I$ is simple, $I\leq_T G$. Hence $I^\omega\leq_T G^\omega\equiv_T G\cdot G$. We conclude that $G\times I^\omega\leq G\cdot G$.
\end{proof}
\section{Commutativity of cofinal types}\label{Sec: proof of the main result}
\begin{definition}
    We say that a class $\mathcal{C}$ of ultrafilters on $\omega$ is a commutative class if for every $U,W\in\mathcal{C}$, $U\cdot W\equiv_T W\cdot U$.
\end{definition}
\begin{example}
    By Dobrinen and Todorcevic, the class $\mathcal{C}_{RP}$ of rapid $p$-points is a commutative class and if $W$ is a rapid $p$-point then $U\cdot W\equiv_T U\times W$.
\end{example}
\begin{example}
    By Milovich \cite{Milovich12}, the class $\mathcal{C}_{P}$ of $p$-points is a commutative class since if $W$ is a $p$-point, then $U\cdot W\equiv_T U\times W\times \omega^\omega$.
\end{example}
 \begin{fact}
     For any three ultrafilters $U,W,Z$, $U\cdot (W\cdot Z)\equiv_{RK} (U\cdot W)\cdot Z$
 \end{fact}
 From the above associativity, we get the following:
 \begin{proposition}
     If $\mathcal{C}$ is a commutative class then so is $\overline{\mathcal{C}}=\{U_1\cdot U_2...\cdot U_n\mid U_1,...,U_n\in \mathcal{C}\}$.
 \end{proposition}
 One use of commutative classes relates to the character of an ultrafilter. Recall that for an ultrafilter $U$, $Ch(U)$ denoted the minimal size of a cofinal subset of $U$. It is clear that if $U\equiv_T U'$ then $Ch(U)=Ch(U')$. Hence we conclude the following: \begin{corollary}
                 Let $\mathcal{C}$ be a commutative class, then for any two ultrafilters $U,V\in\mathcal{C}$, $Ch(U\cdot V)=Ch(V\cdot U)$.
             \end{corollary}

\begin{definition}
    Let $D$ be a cofinal type, define $\mathcal{E}_D=\{U\mid U\cdot U\equiv_T U\times D\}$. 
\end{definition}
For example, $\mathcal{E}_{\{0\}}=\mathcal{E}_{\omega}=\{U\mid U\cdot U\equiv_T U\}$. Also $\mathcal{C}_P\subseteq \mathcal{E}_{\omega^\omega}$.
\begin{claim}
     The class $\mathcal{E}_{\{0\}}\subseteq \mathcal{E}_{\omega^\omega}$.
 \end{claim}
 \begin{proof}
     Just note that $\omega^\omega\leq_T U\cdot U$ for every uniform ultrafilter $U$ and therefore if $U\cdot U\equiv_T U$ then $U\equiv_T U\times \omega^\omega$ and in particular $U\cdot U\equiv_T U\times \omega^\omega$.
 \end{proof}
 \begin{proposition}
     For every cofinal type $D$, $\mathcal{E}_D$ is a commutative class
\end{proposition}
\begin{proof}
    Indeed, for any $U,V\in\mathcal{E}_D$, by Theorem \ref{Thm: Milovich}
    $$V\cdot U\equiv_T V\times U^\omega\equiv_T V\times U\cdot U\equiv_T V\times U\times D$$
    Since the formula above is symmetric for $V,U$, we conclude that $U\cdot V\equiv_T V\cdot U$.
\end{proof}
\begin{proposition}
    The commutative class $\mathcal{E}_{\omega^\omega}$ contains all Tukey-top ultrafilters, $p$-points, stable ordered union ultrafilters and generic ultrafilters for $P(\omega)/I$ where $I^\omega\equiv_T I$ is a simple $\sigma$-ideal 
\end{proposition}
\begin{proof}
    Clearly Tukey top ultrafilters are in any $\mathcal{E}_D$, for $D$ of size at most continuum. $p$-points are in $\mathcal{E}_{\omega^\omega}$ as we have seen by Milovich \cite{Milovich12}, stable ordered union and generic ultrafilters for $P(\omega^\alpha)/I$ are in fact in $\mathcal{E}_{\{0\}}$, by \ref{Cor: sufficient condition for fubini power}. 
\end{proof}
\begin{theorem}
    $\mathcal{E}_{\omega^\omega}$ is closed under Fubini sums.
\end{theorem} \begin{proof}
    Suppose that $\{W,W_0,W_1,...\}\subseteq \mathcal{E}_{\omega^\omega}$.  We need to prove that $$\sum_WW_n\cdot \sum_WW_n\equiv_T \sum_WW_n\times \omega^\omega.$$ Note that $\sum_{W}W_n\geq_T\omega^\omega$, so we will end up getting $\sum_WW_n\cdot \sum_WW_n\equiv_T \sum_WW_n$. It is not hard to see that $$\sum_WW_n\cdot \sum_WW_n\geq_T \sum_WW_n\times \omega^\omega.$$
For the other direction, recall that $\sum_WW_n\cdot \sum_WW_n\equiv_T (\sum_WW_n)^\omega$.    Let us prove that $(\sum_WW_n)^\omega$ is a uniformly below $\mathcal{B}(W,\l W_n\mid n<\omega\r)$. Since $(\sum_WW_n)^\omega$ is complete,  Theorem \ref{thm: uniformly below characterize} can be applied  that get $\sum_{W}W_n\geq_T (\sum_WW_n)^\omega$.

Let $f:W\times \omega^\omega\rightarrow W^\omega$ and $f_n:W_n\times \omega^\omega\rightarrow W_n^\omega$ be monotone and cofinal maps, which exists by the assumption that $W,W_n\in\mathcal{E}_{\omega^\omega}$. We need to define a sequence of monotone and cofinal maps $\l g_X: W\times \prod_{n\in X}W_n\rightarrow (\sum_{W}W_n)^\omega\mid X\in W\r$ such that $(\dagger)$ of definition \ref{Def: uniformly below} holds. 

Let $\l B,\l A_n\mid n\in X\r\r\in W\times \prod_{n\in X}W_n$, by lemma \ref{Lemma: uniformly below on a base}, we may restrict ourselves to sequences satisfying that 
$$(\star)\ \ \text{ for every }n_1,n_2\in X,  \ n_1<n_2\Rightarrow\min(A_{n_1})< \min(A_{n_2}).$$

The first step is to produce $\omega$-many functions in $\omega^\omega$ which are going to be the inputs of $f_n$'s. Fix a partition of $\omega=\biguplus_{l<\omega} Z_l$ such that each $Z_l$ is infinite. Recall that for a set $C$ on natural numbers,  $C(r)$ denotes the $r^{\text{th}}$ element of $C$ in its increasing enumeration. Define $\varphi_{X,i}$ by induction on $i$. For $i=0$, $\varphi_{X,0}(k)=X(Z_0(k))$. Inductively, $\varphi_{X,{i+1}}(k)=\max(\varphi_{X,i}(k),X(Z_{i+1}(k)))$.
\begin{claim}
    If $Y\subseteq X$, then for every $i$, $\varphi_{X,i}\leq \varphi_{Y,i}$.
\end{claim}
\begin{proof}
    Clearly, for every $m<\omega$, $X(m)\leq Y(m)$. So by definition, $\varphi_{X,0}(k)\leq \varphi_{Y,0}$. Suppose this was true for $i$, and let $k<\omega$, then by the induction hypothesis and our first observation, $$\max(\varphi_{X,i}(k),X(Z_{i+1}(k)))\geq\max(\varphi_{Y,i}(k),Y(Z_{i+1}(k)).$$
\end{proof}
The $i$'s function we will use is $h^i_{\l A_n\mid n\in X\r}(k)=\min(A_{\varphi_{X,i}(k)})$. This is well defined as by definition of $\varphi_{X,i}$, $\varphi_{X,i}(k)\in X$. If $Y\subseteq X$, then by the claim $\varphi_{X,i}(k)\leq \varphi_{Y,i}(k)$, and by $(\star)$, $h^i_{\l A_n\mid n\in X\r}(k)\leq h^i_{\l A_n\mid n\in Y\r}(k)$.
Now define 
$g_X(\l B,\l A_n\mid n\in X\r\r)$ by $$ \l\pi_X(\l f(B,h^0_{\l A_n\mid n\in X\r})_m, \l f_n(A_n,h^{n+1}_{\l A_n\mid n\in X\r})_m\mid n\in X\r)\mid m<\omega\r$$
The above seemingly complicated definition is nothing but the composition of the following quite natural monotone cofinal maps:
$$W\times \prod_{n\in X}W_n\overset{(id,\l h^i_*\mid i<\omega\r)}{\longrightarrow} W\times \prod_{n\in X}W_n \times (\omega^\omega)^\omega\rightarrow (W\times\omega^\omega)\times \prod_{n\in X}(W_n \times \omega^\omega)\overset{(f,\l f_n\mid n\in X\r)}{\longrightarrow}$$
$$\overset{(f,\l f_n\mid n\in X\r)}{\longrightarrow} W^\omega\times \prod_{n\in X} W_n^\omega\longrightarrow (W\times \prod_{n\in X} W_n)^\omega\overset{\pi_X^\omega}{\longrightarrow}(\sum_{W}W_n)^\omega$$
So $g_X$ is clearly monotone cofinal as the composition of such functions. To see $(\dagger)$, let $Y\subseteq X$, we ensured that $h^i_{\l A_n\mid n\in X\r}\leq h^i_{\l A_n\mid n\in Y\r}$. Since $f,f_n$ are monotone functions, for any $m<\omega$, and any $n\in Y$, $$f(B,h^0_{\l A_n\mid n\in Y\r})_m\subseteq f(B,h^0_{\l A_n\mid n\in X\r})_m$$ and
$$  f_n(A_n,h^{n+1}_{\l A_n\mid n\in Y\r})_m\subseteq f_n(A_n,h^{n+1}_{\l A_n\mid n\in X\r})_m.$$
By definition of $\pi_X$ and $\pi_Y$, as unions of the relevant sets (see p. 9) we conclude that  $g_Y(\pi_{X,Y}(\l B,\l A_n\mid n\in X\r\r))\geq g_X(\l B,\l A_n\mid n\in X\r\r)$, as wanted. 
\end{proof}
Given a cofinal type $D$, by a (yet) finer handling of coherent sequences of functions, we can define an analog of condition $(\dagger)$, that ensure we can amalgamate monotone cofinal maps $g_X:U\times \prod_{n\in X}V_n\times D\rightarrow \mathbb{P}$, to a monotone cofinal map $g:\sum_{U}V_n\times D\rightarrow \mathbb{P}$. Thus proving an analog of Theorem \ref{thm: uniformly below characterize} that $\sum_UV_n\times D$ is the greatest lower bound of $\{U\times \prod_{n\in X}V_n\times D\mid X\in U\}$ among uniform lower bounds. Then the proof of the previous Theorem generalizes to give the next theorem. Since the proofs are essentially the same, we do not include them. 
\begin{theorem}
Suppose that $D^\omega\equiv_T D$ then $\mathcal{E}_D$ is closed under Fubini sums.
\end{theorem}
We do not know whether the class $\mathcal{E}_{\omega^\omega}$ is provably all uniform ultrafilters on $\omega$. We conjecture that consistently it is not but this will require to produce new examples of non-Tukey top ultrafilters, as it is unknown whether there is a basically generated ultrafilter which is not in the (transfinite) closure of the $p$-points under Fibini sums,  and the only examples for non-Tukey-top non-basically generated ultrafilters are generic ultrafilters for $P(\omega^{\otimes\alpha})/\fin^{\otimes\alpha}$.

        \section{On ultrafilters above $I^\omega$}\label{section: above omega to omega}
        The cofinal type of $\omega^\omega$ came up in several  papers  \cite{Milovich08,Solecki/Todorcevic04} regarding the Tukey order on general ultrafilters. Milovich asked whether there is an ultrafilter $U$ such that $(U,\supseteq)$ is Tukey equivalent to $\omega^\omega$. Let us point out that a negative answer is a straightforward corollary\footnote{Milovich's question appeared only 4 years after Solecki and Todorcevic's result.}
         of Theorem \ref{Thm: todorcevic and solecki}:
        \begin{proposition}\label{Prop: Answer to Milovich}
            There is no non-principal ultrafilter $U$ over $\omega$ such that $(U,\supseteq)\equiv_T\omega^\omega$.
        \end{proposition}
        \begin{proof}
        By Sierpinski \cite{Sierpinski1938}, a non-principal ultrafilter over $\omega$ is a non-measurable set as a subset of $2^\omega$ and in particular non-analytic. An ultrafilter $U$ with the topology inherited from $2^\omega$ is a separable metric space and the set of predecessors is compact and $\omega^\omega$ is a basic analytic order, hence by Theorem \ref{Thm: todorcevic and solecki}, $(U,\supseteq)\not\leq_T \omega^\omega$. 
        \end{proof}
        It turns out that some general problems boil down to being Tukey above $\omega^\omega$. Such a problem is addressed in Theorem \ref{Thm: Dobrinen and todorcevic rapid p-points}, which sets up the equivalence for  $p$-point ultrafilter $U$, between $U\cdot U\equiv_T U$ and $U\geq_T \omega^\omega$. 
        This was generalized in Theorem \ref{Thm: TomNatahsa equivalence for product} which ensures that for a general ultrafilter $U$,  $U\cdot U\equiv_T U$ is equivalent to the existence of \textit{some} $I\subseteq U^*$ such that $U$ has the $I$-p.i.p and $U\geq_T I^\omega$. In the first part of this section, we tighten the connection between the $I$-p.i.p and being above $I^\omega$ for a deterministic ideal $I$. Then, in the second subsection, we shall restrict our attention to $I=\omega$.
        \subsection{The case of a deterministic ideal $I$}
        There is a slight difference between the type of equivalence in Theorem \ref{Thm: Dobrinen and todorcevic rapid p-points} to the equivalence $U\cdot U\equiv_T U$ for $p$-points and the general one in \ref{Thm: TomNatahsa equivalence for product}. Indeed, in the latter, the ideal $I$ can vary. Hence it is unclear in general if for a fixed $I$, the following is true: for any ultrafilter $U$  which has the $I$-p.i.p, $U\cdot U \equiv_TU$ iff $U\geq_TI^\omega$.
        Let us note first that such equivalence holds for simple ideals (and therefore also for deterministic ideals).
        \begin{proposition}\label{Prop: equivalence for simple ideals}
            Suppose that $I$ is simple. Then for any ultrafilter $U$ which has the $I$-p.i.p, $U\cdot U\equiv_T U\times I^\omega$. Therefore, the following are equivalent:
            \begin{enumerate}
                \item $U\cdot U\equiv_T U$.
                \item $U\geq_T I^\omega$.
            \end{enumerate}
        \end{proposition}
\begin{proof} 
By Theorem \ref{Thm: Milovich}, $U\cdot U\equiv_T U^\omega$. Since $I$ is simple, and $I\subseteq U^*$, $U\geq_T I$ and in particular $U\cdot U\equiv_TU^\omega\geq_T U\times I^\omega$. The other direction follows from the $I$-p.i.p of $U$ and Proposition \ref{Prop: TomNatasha bound with p.i.p}.
Now to see the equivalence,
$(2)\Rightarrow (1)$ follows from Theorem \ref{Thm: Sufficint condition for product from TomNatash2}, and $(1)\Rightarrow (2)$ follows from the first part as $U\cdot U\equiv_T U\times I^\omega\leq_T U\leq_T U\cdot U$.
\end{proof}
    The previous proposition is somewhat different from  \cite[Thm. 1.18]{TomNatasha2}, as here $I$ is a fix simple ideal.
    
    Our next objective it to study the class of ultrafilters which are Tukey above $I^\omega$. The next theorem shows that for deterministic $I$'s this class extends the class of ultrafilters which do not have the $I$-p.i.p.
    
    
            
\begin{theorem}
    Suppose that $\fin\subseteq I\subseteq U^*$ is deterministic, then if $U\not\geq_T I^\omega$, then $U$ has the $I$-p.i.p
\end{theorem}
\begin{proof}
    Let us verify the equivalent condition in Proposition \ref{Prop: equivalent condition}, let $\l A_n\mid n<\omega\r$ be a partition of $\omega$ such that $A_n\notin U$. We need to find $X\in U$ such that $X\cap A_n\in I$ for every $n$. Without loss of generality, suppose that $A_n\in I^+$ for every $n$. Since $\fin\subseteq I$, $A_n$ is infinite and we can find a bijection $\pi:\omega\leftrightarrow \omega\times\omega$ such that $\pi''A_i=\{i\}\times \omega$. Let $W=\pi_*(U)$ be the Rudin-Keisler isomorphic copy of $U$. For each $n<\omega$, consider the ideal $I_n=\pi_*(I\cap P(A_n))$ on $\{n\}
    \times \omega$. By Proposition \ref{Prop: deterministic invariates}, $I\cap P(A_n)$ is a deterministic  and since $\pi\restriction A_n$ is one-to-one $I_n=\pi_*(I\cap P(A_n))$ is deterministic. It follows by \ref{Prop: product of deterministic is also} $\sum_{\fin}I_n$ is deterministic.
    Moreover, by Fact \ref{Fact: Positive deterministic ideal is Tukey equivalent}, $I\equiv_T I\cap P(A_n)\equiv_T I_n$ and therefore $$I^\omega\equiv_T\prod_{n<\omega}I_n\equiv_T\sum_{\fin}I_n$$ 
    Since $U\not\geq_T I^\omega$, 
 $W\not\geq_T\sum_{\fin}I_n$. Since $\sum_{\fin}I_n$ is deterministic, it follows that  $\sum_{\fin}I_n\not\subseteq W^*$. Thus, there is $X'\in \sum_{\fin}I_n\cap W$. Namely, for all but finitely many $n$'s, $(X')_n\in I_n$. Since each $A_i\notin U$, we may assume that for every $n$, $(X')_n\in I_n$. 
  Let $X=\pi^{-1}[X']$, then for every $n<\omega$, $X\cap A_n\in I$ as $\pi''X\cap A_n=\{x\}\cap(X)_n\in I_n$.
\end{proof}

The proof of the above actually gives the following:
        \begin{corollary}
            Suppose that $U$ does not have the $I$-p.i.p, then there is $W\subseteq \omega\times\omega$ such that $W\equiv_{RK}U$ and $\sum_{\fin}I_n\subseteq W$, and each $I_n\equiv_T I$. In particular, the Tukey type of $I^\omega$ is realized as a sub-ideal of $U^*$. 
        \end{corollary}
Taking $I=\fin$ in the above we obtain the following corollary
\begin{corollary}
    
            Suppose that $U$ is a non-principal ultrafilter such that $U\not\geq_T \omega^\omega$ then $U$ is a $p$-point.
\end{corollary}
        As a corollary, we see that in Proposition \ref{Prop: equivalence for simple ideals} and therefore also in Theorem \ref{Thm: Dobrinen and todorcevic rapid p-points}, the $I$-p.i.p assumption is somewhat redundant. 
\begin{corollary}
    If $I\subseteq U^*$ is deterministic then the following are equivalent:
    \begin{enumerate}
        \item $U\equiv_{T}U\cdot U$ or $U$ does not have the $I$-p.i.p.
        \item $U\geq_T I^\omega$ 
    \end{enumerate}
\end{corollary}
        
\subsection{Ultrafilters above $\omega^\omega$}        
As observed by Dobrinen and Todorcevic, rapid ultrafilters form a subclass of those which are Tukey above $\omega^\omega$. By Miller \cite{MillerNoQPoints}, in the Laver model where the Borel conjecture holds\footnote{ i.e. the model obtained by adding $\omega_2$-many Laver reals to a model of $CH$.}, there are no rapid ultrafilters. On the other hand, there are always ultrafilters which are above $\omega^\omega$, namely tukey-top ultrafilters. The question regarding the existence (in $ZFC$) of non-Tukey top ultrafilters is a major open problem. Hence we do not expect to have a simple $ZFC$ construction of a non-Tukey-top ultrafilter which is above $\omega^\omega$.

The result from the previous section entails some drastic consistency results regarding the class of ultrafilters above $\omega^\omega$:
        \begin{corollary}
            Suppose that there are no $p$-points, then every ultrafilter is above $\omega^\omega$. 
        \end{corollary}
        Models with no $p$-points were first constructed by Shelah \cite{ShelahProper} and later by Chudonsky and Guzman \cite{Chodounsky2017ThereAN}. 
        
            By yet another result of Shelah, in the Miller model \cite{Miller2}, which is obtained by countable support iteration of the superperfect tree forcing of length $\omega_2$ over a model of CH, every $p$-point is generated by $\aleph_1$-many sets. It is known that $\mathfrak{d}=\mathfrak{c}$ holds in that model. Therefore, every $p$-point is generated by less than $\mathfrak{d}$-many sets and in particular not above $\omega^\omega$. 
        \begin{corollary}
              It is consistent that  $p$-points are characterized by not being above $\omega^\omega$.
        \end{corollary}

        Focusing on models less extreme than the ones above,  we may be interested in those $p$-point which are above $\omega^\omega$. The purpose of this section is to address the question raised in \cite{TomNatasha2} whether rapid $p$-point are exactly those $p$-points which are above $\omega^\omega$ (The dashed area in Figure \ref{Figure 1}). As a warm-up, let us note that there are always non-rapid ultrafilters above $\omega^\omega$. To see this, we need the following result \cite[Thm. 4]{MillerNoQPoints}:
    \begin{proposition}[Miller]
            For any two ultrafilters $U,V$ on $\omega$, $U\cdot V$ is rapid iff $V$ is rapid.
        \end{proposition}By results of Choquet \cite{choquet1968deux}, there are always non-rapid ultrafilters. Taking any such $U$, $U\cdot U$ is certainly above $\omega^\omega$ and by the above result of Miller, it is non-rapid.
        
    \begin{corollary}
        
        There is a non-rapid ultrafilter which is Tukey above $\omega^\omega$.\
    \end{corollary}
    Note that the ultrafilter we constructed in the previous proof in not a $p$-point as it is a product.
\vskip 0.7 cm
            \begin{center}
            {\tiny Figure \ref{Figure 1}}.
            \vskip 0.2 cm
\begin{tikzpicture}\label{Figure 1}
\draw[black, very thick] (0,0) rectangle (10,7);
\draw[black, fill=gray, opacity=0.3, very thick] (0,4.5) rectangle (10,7);

\draw[black, fill=gray!60, opacity=0.5, very thick] (0,0) rectangle (10,2.5);
\draw[black, fill=gray, opacity=0.2, very thick] (7,0) rectangle (10,1);
\draw[black, fill=gray, opacity=0.5, very thick] (0,0) rectangle (3,7);
\fill [pattern={Lines[angle=45,distance=7pt]},pattern color=black] 
    (3,0) rectangle (7,2.5);
    \fill [pattern={Lines[angle=45,distance=7pt]},pattern color=black] 
    (7,1) rectangle (10,2.5);
    \node at (5,5.5){Tukey-top};
    \node at (5,1.5){p-point};
    \node at (1.5,3.5){rapid};
    \node at (8.5,0.5){$\not\geq_T\omega^\omega$};
    \end{tikzpicture}
    \end{center}
    \vskip 0.7 cm

        The real issue is to construct a $p$-point which is not rapid but still above $\omega^\omega$. To do that, let us introduce the class of $\alpha$-almost rapid ultrafilters. 
            
                Given a function $f:\omega\rightarrow\omega\setminus\{0\}$ such that $f(0)>0$. We denote by $exp(f)(0)=f(0)$ and $$exp(f)\big(n+1\big)=f\big(exp(f)(n)\big)=f(f(f(f...f(0)..))).$$ We define the $n^{\text{th}}$ $f$-exponent function, $$exp_0(f)=f\text{ and }exp_n(f)=exp(exp_{n-1}(f)).$$ 
                This definition continuous transfinitely for every $\alpha<\omega_1$:
                $$exp_{\alpha+1}(f)=exp(exp_{\alpha}(f)).$$ For limit $\delta<\omega_1$, we fix some increasing cofinal sequence $\l \delta_n\mid n<\omega\r$ in $\delta$, and let $$exp_\delta(f)(n)=\max\{exp_{\delta_n}(f)(n), exp_{\delta}(f)(n-1)+1\}.$$
                \begin{lemma}\label{Lemma: Prop of exp}
                  Let $f,g:\omega\rightarrow\omega$ be increasing functions. \begin{enumerate}
                      \item For every $\alpha<\omega_1$, $exp_\alpha(f)$ is increasing.
                      \item  If $f\leq g$ then for every $\alpha<\omega_1$, $exp_\alpha(f)\leq exp_\alpha(g)$.
                      \item For every $\alpha<\beta<\omega_1$, $exp_\alpha(f)<^* exp_\beta(f)$.
                  \end{enumerate}  
                \end{lemma}
                \begin{proof}
                For $(1)$, we proceed by induction. For $\alpha=0$, $exp_0(f)=f$ is increasing. Suppose $exp_\alpha(f)$ is increasing, then for every $n<\omega$, $exp_\alpha(f)(n)>n$. For $\alpha+1$, let $n<\omega$. Since $exp_\alpha(f)$ is increasing, $$exp_{\alpha+1}(f)(n+1)=exp_{\alpha}(f)(exp_{\alpha+1}(f)(n))>exp_{\alpha+1}(f)(n).$$ For limit $\delta$,  is clear from the definition that $exp_\delta(f)$ is increasing. Also (2) is proven by induction. The base case is $exp_0(f)=f\leq g=exp_0(g)$. Suppose this was true for $\alpha$, and let us prove by induction on $n<\omega$. The base again is $$exp_{\alpha+1}(f)(0)=exp_\alpha(f)(0)\leq exp_\alpha(g)(0)\leq exp_{\alpha+1}(g)(0)$$
                Suppose that $exp_{\alpha+1}(f)(n)\leq exp_{\alpha+1}(g)(n)$, then by $(1)$ and the induction hypothesis
                $$exp_{\alpha+1}(f)(n+1)=exp_{\alpha}(f)(exp_{\alpha+1}(f)(n))\leq exp_{\alpha}(f)(exp_{\alpha+1}(g)(n))$$
                $$\leq exp_\alpha(g)(exp_{\alpha+1}(g)(n))=exp_{\alpha+1}(g)(n+1)$$
                At limit stages $\delta$, by the induction hypothesis, $$exp_\delta(f)(n)=exp_{\delta_n}(f)(n)\leq exp_{\delta_n}(g)(n)=exp_{\delta}(g)(n).$$
                Finally, $(3)$ is a standard diagonalization argument.
                \end{proof}
        \begin{definition}\label{Def: almost-rapid}
            For $\alpha<\omega_1$, we say that an ultrafilter $U$ is  \textit{$
            \alpha$-almost-rapid} if for every function $f\in\omega^\omega$ there is $X\in U$ such that $exp_\alpha(f_X)\geq^*f$, where $f_X$ is the increasing enumeration of $X$. 
        \end{definition}
        \begin{remark}
            By strengthening the  above definition, we may require that $exp_{\alpha}(f_X)\geq f$. However, this strengthening turns out to be an equivalent definition.
        \end{remark}
            Note that $0$-almost-rapid is just rapid, and by $(3)$ of the previous lemma, if $\beta\leq\alpha$ then $\beta$-almost-rapid implies $\alpha$-almost-rapid. We call $U$ \textit{almost-rapid} if it is $1$-almost-rapid.
            \begin{proposition}
                If $U$ is $\alpha$-almost-rapid implies $U\geq_T\omega^\omega$
            \end{proposition}
\begin{proof}
Consider the map $X\mapsto exp_\alpha(f_X)$. We claim that it is monotone and cofinal.
   First, suppose that $X\subseteq Y$, then the natural enumerations $f_X,f_Y$ of $X,Y$ (resp.) satisfy $f_X\geq f_Y$. Then by Lemma \ref{Lemma: Prop of exp}(3) $exp_\alpha(f_X)\geq exp_\alpha(f_Y)$. The map above is cofinal by the $\alpha$-almost rapidness of $U$.
\end{proof}
Rapid ultrafilters are characterized by the following property \cite{MillerNoQPoints}: An ultrafilter over $\omega$ is rapid if and only if for every sequence $\l P_n \mid n<\omega\r$ of finite subsets of $\omega$, there is $X\in U$ such that for every $n<\omega$, $|X\cap P_n|\leq n$. The proposition below provides an analogous characterization of almost-rapid ultrafilters.
    \begin{proposition}
        The following are equivalent:
        \begin{enumerate}
            \item $U$ is almost-rapid.
            \item For any sequence $\l P_n\mid n<\omega\r$ of sets, such that $P_n$ is finite, there is $X\in U$ such that for each $n<\omega$, $exp(f_X)(n-1)\geq |X\cap P_n|$ (where $exp(f_X)(-1)=0$).
            \end{enumerate}
    \end{proposition}
    \begin{proof}

        \underline{$(1)\Rightarrow(2)$:} Suppose that $U$ is almost rapid, and let $\l P_n\mid n<\omega\r$ be a sequence as above. Let $f(n)=\max(P_n)+1$. By $(1)$, there is $X\in U$ which is obtained by almost-rapidness, namely $exp(f_X)(0)=f_X(0)=\min(X)> f(0)>\max(P_0)$ and therefore $X\cap P_0=\emptyset$. Next, $exp(f_X)(1)=f_X(f_X(0))>f(1)>\max(P_1)$ hence $|P_1\cap X|\leq f_X(0)=exp(f_X)(0)$. In general $f_X(exp(f_X)(n))=exp(f_X)(n+1)>f(n+1)>\max(P_{n+1})$ and therefore $|X\cap  P_{n+1}|\leq exp(f_X)(n)$.
        
        \underline{$(2)\Rightarrow (1)$:} Let $f$ be any function. Let $P_n=f(n)$. Then by $(2)$, there is $X$ such that $|X\cap f(n)|\leq exp(f_X)(n-1)$. In particular, $X\cap f(0)=\emptyset$ and therefore $exp(f_X)(0)=f_X(0)=\min(X)\geq f(0)$. In general, $|X\cap f(n)|\leq exp(f_X)(n-1)$ and therefore $f(n)<f_X(exp(f_X)(n-1))=exp(f_X)(n)$. 
    \end{proof}
    \begin{theorem}\label{Thm: Separating almost rapid from rapid}
        Assume CH. Then there is a $p$-point which is almost-rapid but not rapid
    \end{theorem}
    \begin{proof}
        Let $P_n=\{1,...,2^n\}$. Let $I=\{A\subseteq\omega\mid \exists k\forall n, \ |A\cap P_n|\leq k\cdot n\}$. Then $I$ is a proper ideal on $\omega$. Suppose that $\l P_n\mid n<\omega\r$ is not a counterexample for $U$ being rapid, then there is a set $X\in U$ such that $|X\cap P_n|\leq n$ for every $n$ and therefore $X\in I$. Hence, as long as we have $U\subseteq I^+$, $U$ will not be rapid. 
        Note that
        $$A\in I^+\text{ iff for every }k, \ \text{there is }n_k\text{ such that }|A\cap P_{n_k}|>kn_k.$$ Or equivalently, $n\mapsto |A\cap P_n|$ is not asymptotically bounded by a linear function of $n$. The following is the key lemma for our construction:
        \begin{lemma}
            Suppose that $\l A_n\mid n<\omega\r\subseteq I^+$ is $\subseteq$-decreasing, and $f:\omega\rightarrow\omega$. Then there is $B\subseteq \omega$ such that
            \begin{enumerate}
                \item $B\subseteq^* A_n$ for every $n$.
                \item $B\in I^+$.
                \item $exp(f_B)>f$.
            \end{enumerate} 
        \end{lemma}
        \begin{proof}
            Suppose without loss of generality that $f$ is increasing. In particular, $f(k)\geq k$. Consider $f(1)$, find $2<n_1$ so that $$  |A_1\cap P_{n_1}|>(f(1)+2)\cdot n_1$$ such an $n_1$ exists as $A_1$ is positive and taking $k=f(1)+2$. Find $a_0,...,a_{n_1+1}$ such that
            \begin{enumerate}
                \item $f(0),n_1+1<a_0<a_1<....<a_{n_1+1}$.
                \item $a_{n_1+1}>f(1)$.
            \item $a_0,a_1,...,a_{n_1+1}\in A_1\cap P_{n_1}$.  
            \end{enumerate}   It is possible to find such elements as $$|A_1\cap P_{n_1}\setminus \{0,...,n_1+1\}|>(f(1)+2)n_1-(n_1+2)\geq 3n_1-n_1-2=2n_1-2\geq n_1+1.$$ So there are $n_1+1$ elements in $A_1\cap P_{n_1}$ greater than $n_1+1$. Since $|A_1\cap P_{n_1}|>f(1)$, we can also make sure that the $n_1+1$ element we choose is above $f(1)$. This way, we have guaranteed that:
            \begin{enumerate}
                \item $f(0)<f(1)<a_0$.
                \item $a_{a_0}$ was not defined yet (!), but as long as the sequence is increasing, $a_{a_0}>f(1)$.
                \item For $k=1$, there is $n_1$ such that $|B\cap P_{n_1}|>n_1$
            \end{enumerate}
            Now consider $f(2)$ and find $n_2>2,a_{n_1+1}$ so that $$|A_2\cap P_{n_2}|>(f(2)+2)(a_{n_1+1}+1)n_2,$$ we find
            \begin{enumerate}
                \item $n_1+1+2n_2+1<a_{n_1+2}<...<a_{n_1+1+2n_2+1}$.
                \item $f(2)<a_{n_1+1+2n_2+1}$.
                \item $a_{n_1+2},...,a_{n_1+1+2n_2+1}\in A_2\cap P_2$.
            \end{enumerate}
            This is possible to do since
            $$|A_2\cap P_{n_2}\setminus \{0,...,n_1+2n_2+2\}|>(f(2)+2)(a_{n_1+1}+2)n_2-(n_1+2n_2+3)>$$
            $$>8n_2-(3n_2+3)=6n_2-3>2n_2+1$$
            So we can find $a_{n_1+2},...,a_{n_1+1+2n_2+1}$ above $n_1+1+2n_2+1$  (and therefore also above $a_{n_1+1}$). We can also make sure that the last element we pick is above $f(2)$. This way we ensured the following:
            \begin{enumerate}
                \item As we observed, $a_{a_0}$ was not defined in the first round (and might not be defined in the second round as well) and therefore (a possibly future) $a_1':=a_{a_0}>n_1+1+2n_2+1$. Thus a future $a'_2:=a_{a_{a_0}}>a_{n_1+1+2n_2+1}>f(2)$.
                \item For $k=2$, there is $n_2$ such that $|B\cap P_{n_2}|>2n_2$.
            \end{enumerate}
            In general suppose we have defined $n_1<n_2<....<n_k$ 
            and $a_0,...,a_{\sum_{i=1}^kin_i+1}$ and such that $a'_{k-1}>\sum_{i=1}^kin_i+1$. Then we find $n_{k+1}>k+1,a_{\sum_{i=1}^kin_i+1}$ such that $|A_{k+1}\cap P_{n_{k+1}}|>3(k+1)(f(k+1)+1)n_{k+1}$. We now define $$a_{(\sum_{i=1}^kin_i+1)+1},a_{(\sum_{i=1}^kin_i+1)+2}...,a_{\sum_{i=1}^{k+1}in_i+1}$$ (that is $(k+1)n_{k+1}+1$ many elements) so that:
            \begin{enumerate}
                \item $\sum_{i=1}^{k+1}in_i+1<a_{(\sum_{i=1}^kin_i+1)+1}<...<a_{\sum_{i=1}^{k+1}in_i+1}$,
                \item $a_{\sum_{i=1}^{k+1}in_i+1}>f(k+1)$.
                \item $a_{(\sum_{i=1}^kin_i+1)+1},...,a_{\sum_{i=1}^{k+1}in_i+1}\in A_{k+1}\cap P_{n_{k+1}}$.
            \end{enumerate}
            To see that such $a$'s exists, 
note that $$|A_{k+1}\cap P_{n_{k+1}}\setminus\{0,...,\sum_{i=1}^{k+1}in_i+1\}|>3(k+1)(f(k+1)+1)n_{k+1}-(\sum_{i=1}^{k+1}in_i+1)-1$$
$$>3(k+1)(f(k+1)+1)n_{k+1}-((k+1)n_{k+1}+1)-(\sum_{i=1}^{k}in_i+1)-1$$
$$(k+1)(3f(k+1)+3)n_{k+1}-2((k+1)n_{k+1}+1)$$
$$>(k+1)(3f(k+1)+1)n_{k+1}>(k+1)n_{k+1}+1$$
Hence we can find $(k+1)n_{k+1}+1$-many elements in $A_{k+1}\cap P_{k+1}$ above $\sum_{i=1}^{k+1}in_i+1$. Also, since $|A\cap P_n|>f(k+1)$ we can make sure that $a_{\sum_{i=1}^{k+1}in_i+1}>f(k+1)$.
            This way we ensure that:
            \begin{enumerate}
                \item Since $a_{a'_{k-1}}$ was not previously defined in previous rounds, $a'_{k}:=a_{a'_{k-1}}>\sum_{i=1}^{k+1}in_i+1$ and $a_{a'_k}$ has not been defined yet. Hence a future $a'_{k+1}:=a_{a'_k}>f(k+1)$.
                \item $|B\cap P_{n_{k+1}}|>(k+1)n_{k+1}+1$.
            \end{enumerate}            
            Set $B=\{a_n\mid n<\omega\}$. So by the construction, for every $k$ there is $n_k$ such that $|B\cap P_{n_k}|>kn_k$.  Hence $B\in I^+$. Also, note that $f_B(n)=a_n$ since the $a_n$'s are increasing. By the construction and definition of $exp(f)$, $exp(f_B)(n)=a'_n>f(n)$. Finally, note that for each $n$, there is $k$ such that for every $k'\geq k$, $a_{k'}\in A_m$ for some $m\geq n$. Since the sequence of $A_n$'s is $\subseteq$-decreasing, $a_{k'}\in A_n$. We conclude that $B\setminus A_n\subseteq \{a_0,...,a_k\}$.
        \end{proof}
        Now for the construction of the ultrafilter. Enumerate $P(\omega)=\l X_\alpha\mid \alpha<\omega_1\r$, and $P(\omega)^{\omega}=\l \vec{A}_\alpha\mid \alpha<\omega_1\r$ such that each sequence in $P(\omega)^\omega$ appears cofinaly many times in the enumeration. Also enumerate $\omega^\omega=\l \tau_\alpha\mid\alpha<\omega_1\r$.
        We define a sequence of filters $V_\alpha$ such that:
        \begin{enumerate}
            \item $\beta<\alpha\Rightarrow V_{\beta}\subseteq V_\alpha$.
            \item $V_\alpha\subseteq I^+$.
            \item Either $X_\alpha$ or $\omega\setminus X_\alpha\in V_{\alpha+1}$.
            \item There is $X\in V_{\alpha+1}$ such that $\tau_\alpha<exp(f_X)$.
            \item If $\vec{A}_\alpha\subseteq V_{\alpha}$ then there is a pseudo-intersection $A\in V_{\alpha+1}$.  
        \end{enumerate}
        Let $V_0=I^*$. At limit steps $\delta$ we define $V_\delta=\bigcup_{\beta<\delta}V_\beta$. It is clear that $(1)-(5)$ still holds at limit steps. At successors, given $V_\alpha$, since we have only performed countably many steps so far, there are sets $B_n\in V_\alpha$ such that $V_\alpha=I^*[\l B_n\mid n<\omega\r]$ where $B_n$ is $\supseteq$-decreasing. If either $X_\alpha$ or $\omega\setminus X_\alpha$ is already in $V_\alpha$, we ignore it. Otherwise, we must also have $V_\alpha[X_\alpha]\subseteq I^+$.
        If $\vec{A}_\alpha\not\subseteq V_\alpha[X_\alpha]$ ignore it. Otherwise, enumerate the set
        $$\{B_n\cap X_\alpha\mid n<\omega\}\cup \{\vec{A}_\alpha(n)\mid n<\omega\}\subseteq V_\alpha[X_\alpha]$$ by $\l B'_n\mid n<\omega\r$ and let $C_n=\cap_{m\leq n}B'_m$. We apply the previous lemma to the sequence $\l C_n\mid n<\omega\r$, and $\tau_\alpha$ to find $A^*\subseteq \omega$ such that:
        \begin{enumerate}
            \item $A^*\in I^+$.
            \item $exp(f_{A^*})>\tau_\alpha$.
            \item $A^*\subseteq^* C_n$ for every $n$.
        \end{enumerate}
        Since for every $n<\omega$, there is $n'$ such that $C_{n'}\subseteq \vec{A}_\alpha(n)\cap B_n$, $A^*\subseteq^* \vec{A}_\alpha(n)$, namely $A^*$ is a pseudo intersection of both $\l B_n\mid n<\omega\r$ and $\vec{A}_\alpha$. Also, $A^*$ is a positive set with respect to the ideal $V_\alpha[X_\alpha]$. Otherwise, there is some $A\in I^*$ and $B_n$ such that $A^*\cap (A\cap B_n\cap X_\alpha)=\emptyset$. But then $(A^*\cap B_n\cap X_\alpha)\cap A=\emptyset$ which implies that $A^*\cap B_n\cap X_\alpha\in I$. However, $A^*\subseteq^* B_n\cap X_\alpha$, which implies that $A^*\in I$, contradicting property $(1)$ above in the choice of $A^*$. Hence we can define $V_{\alpha+1}=V_\alpha[X_\alpha,A^*]$ and $(1)-(5)$ hold. 
        
        This concludes the recursive definition. The ultrafilter witnessing the theorem is defined by $V^*=\bigcup_{\alpha<\omega_1}V_\alpha$.
        \begin{proposition}
            $V^*$ is a non-rapid almost-rapid $p$-point ultrafilter. \end{proposition}
        \begin{proof}
            $V^*$ is an ultrafilter since for every $X\subseteq\omega$, there is $\alpha$ such that $X=X_\alpha$ and so either $X_\alpha$ or $\omega\setminus X_\alpha$ are in $V_{\alpha+1}\subseteq V^*$. Also $V^*$ is a $p$-point since if $\l A_n\mid n<\omega\r\subseteq V^*$ then there is $\alpha<\omega_1$ such that $\l A_n\mid n<\omega\r\subseteq V_\alpha$ and by the properties of the enumeration there is $\beta>\alpha$ such that $\vec{A}_\beta=\l A_n\mid n<\omega\r$. This means that in $V_{\beta+1}$ there is a pseudo intersection for the $A_n$'s.
It is non-rapid as $V^*\subseteq I^+$ and, in fact, the sequence $P_n$ witnesses that it is non-rapid. Finally, it is almost rapid since for any function $\tau:\omega\rightarrow\omega$, there is $\alpha$ such that $\tau=\tau_\alpha$ and therefore in $V_{\alpha+1}$ there is a set $X$ such that $exp(f_X)>\tau$.
        \end{proof}

    \end{proof}
    \begin{corollary}
        It is consistent that there is a $p$-point which is not rapid but still above $\omega^\omega$.
    \end{corollary}
    \begin{remark}
        CH is not necessary in order to obtain such an ultrafilter, since we can, for example, repeat a similar argument in the iteration of Mathias forcing after we forced the failure of $CH$ and obtain such an ultrafilter. In fact, we conjecture that the construction of Ketonen \cite{Ketonen1976} of a $p$-point from $\mathfrak{d}=\mathfrak{c}$ can be modified to get a non-rapid almost-rapid $p$-point.
    \end{remark}

 \section{questions}\label{Sec:Question}
 We collect here some problems which relate to the work of this paper. The first batch of questions regards the Tukey-type of sums of ultrafilters:
 \begin{question}
    If $\mathbb{P}$ is below $\mathcal{B}(U,V_\alpha)$ does it imply that $\mathbb{P}$ is uniformly below $\mathcal{B}(U,V_\alpha)$? What about the case where $\mathbb{P}$ is an ultrafilter?
\end{question}
\begin{question}
    Is it true in general that $\sum_{U}V_\alpha=\inf\mathcal{B}(U,V_\alpha)$?  
\end{question}
\begin{question}
Is there a nice characterization for the Tukey-type of $\sum_UV_\alpha$ if we assume that $V_0\geq_T V_1\geq_T V_2....$? 
\end{question}
\begin{question}
    Is there a nice characterization for the Tukey-type of $\sum_UV_\alpha$ when the sequence of $V_\alpha$'s is discrete? 
\end{question}
How much of the theory developed here generalizes to measurable cardinals? more concretely:
\begin{question}
Does Lemma \ref{Lemma: increasing v_n's case} hold true for $\sigma$-complete ultrafilters over uncountable cardinals?    
\end{question}

The next type of questions relate to the $I$-pseudo intersection property
\begin{question}
    What is the characterization of the $I$-p.i.p property in terms of Skies and Constellations of ultrapowers from \cite{PURITZ1972215}?
\end{question} 

 \begin{question}
     Is the equivalence of Proposition \ref{Prop: equivalence for simple ideals} true for every ideal $I$?
 \end{question}

 The following addresses the possibility of full commutativity of cofinal types between ultrafilters on $\omega$ and on different cardinals.
 \begin{question}
 Is it provable that for any two ultrafilters $U,V$, $U\cdot V\equiv_T V\cdot U$?\end{question}
\begin{question}
     Is it true that for every two ultrafilters $U,V$ on any cardinals, $U\cdot V\equiv_T V\cdot U$?
 \end{question}
 Let us note that if $U,V$ are $\lambda$ and $\kappa$ ultrafilters respectively, then this holds. To see this, first note that if $\kappa=\lambda\geq_T\omega$, then this is a combination of the results from \cite{TomNatasha} and this paper's main result. Without loss of generality, $\lambda<\kappa$. In which case, $U\cdot V=U\times V^\lambda$. Since $V$ is $\lambda^+$-complete, it is not had to see that $V^\lambda\equiv_T V$, hence $U\cdot V\equiv_T V\cdot U$. On the other hand, since $2^\lambda<\kappa$, every set $X\in V\cdot U$, contained a set of the form $X\times Y$ for $X\in V$ and $Y\in U$. It follows that also $V\cdot U\equiv_T U\times V\equiv_T U\cdot V$.  So the question above is interesting in case we drop the completeness assumption on the ultrafilters.

the results in \S~\ref{section: above omega to omega} suggest that the Tukey-type of $I^\omega$ as an important role in the analysis of non Tukey-top ultrafilters. Note that the same argument which worked for $\omega^\omega$ in Proposition \ref{Prop: Answer to Milovich} works for $I^\omega$, hence leading to the conclusion that if $I$ is an analytic ideal the no non-principal ultrafilter can be Tukey equivalent to $I^\omega$. The following question is more open-ended:
\begin{question}
    How can we force different values of the Tukey-type of $I^\omega$, when $I\not\equiv_T \omega^\omega$?
\end{question}
 Finally, we present a few questions regarding the new class of $\alpha$-almost-rapid ultrafilters.
\begin{question}
Is it true that for every $\alpha<\beta<\omega_1$, the class of $\alpha$-almost-rapid ultrafilters is consistently strictly included in the class of $\beta$-almost-rapid ultrafilters?  
\end{question} 
We conjecture a positive answer to this question and that similar methods to the one presented in Theorem \ref{Thm: Separating almost rapid from rapid} under CH should work.
\begin{question}
    Does $\mathfrak{d}=\mathfrak{c}$ imply that there is a $p$-point which is almost-rapid but not rapid?
\end{question}
Ultimately we would like to understand if such ultrafilters exist in $ZFC$:
    \begin{question}
        Is it consistent that there are no almost-rapid ultrafilters?
    \end{question}
    Following Miller, a natural model would be adding $\aleph_2$-many Laver reals. However, the current argument does not immidietly rule out $\alpha$-almost rapid ultrafilters. 
    \subsection*{Acknowledgment} I would like to thank Gabriel Goldberg for stimulating discussions and for suggesting an example of ultrafilters which eventually led to Proposition \ref{prop: example in full support}. I would also like to thank Slawomir Solecki and Stevo Todorcevic for answering some questions regarding their work on the Tukey order. 
    Finally, a special gratitude goes to Natasha Dobrinen for introducing me to the subject,  a task which took a great amount of her precious time, and for her comments on early drafts of this paper.  
    

     
 
\bibliographystyle{amsplain}
\bibliography{ref}
\end{document}